\documentclass[10pt]{article}
\usepackage[latin1]{inputenc}
\usepackage{a4wide}
\usepackage{amsmath}
\usepackage{amssymb}
\usepackage{amsfonts}
\usepackage{amsthm}
\usepackage{url}
\usepackage{graphicx}
\usepackage{nicefrac}
\usepackage{units}
\usepackage{url}
\usepackage{comment}

\usepackage{tikz}
\usetikzlibrary{trees,arrows}
\usepackage{caption} 

\usepackage{subfigure}

\newcommand{\E}{\ensuremath{\mathbb{E}}}
\newcommand{\V}{\ensuremath{\mathrm{Var}}}
\newcommand{\CV}{\ensuremath{\mathrm{Cov}}}
\newcommand{\Prob}{\ensuremath{\mathbb{P}}}

\newcommand{\R}{\ensuremath{\mathbb{R}}}
\newcommand{\N}{\ensuremath{\mathbb{N}}}
\newcommand{\bo}{\ensuremath{\mathrm{O}}}
\newcommand{\Ind}{\ensuremath{\textbf{1}}}

\newtheorem{thm}{Theorem}[section]
\newtheorem{prop}[thm]{Proposition}

\newtheorem{lem}[thm]{Lemma}

\def\No{{\cal N}}

\begin{document}
\title{\bf A Limit Theorem for Radix Sort and Tries with Markovian Input}
\author{Kevin Leckey and Ralph Neininger\thanks{This author's research was supported by DFG grant Ne 828/2-1.}\\Institute for Mathematics\\
J.W.~Goethe University Frankfurt\\
60054 Frankfurt am Main\\
Germany\\
\{leckey, neiningr\}{@}math.uni-frankfurt.de
   \and Wojciech Szpankowski
\thanks{This author's work was partially done when visiting
J.W.~Goethe University Frankfurt a.M. on the Alexander von Humboldt
research award. This work was also supported by
NSF Center for Science of Information (CSoI) Grant CCF-0939370,
and in addition by NSA Grant 130923, and NSF Grant DMS-0800568,
NIH Grant 1U01CA198941-01, and the MNSW grant DEC-2013/09/B/ST6/02258. 
W. Szpankowski is also with the Faculty of Electronics, Telecommunications and
Informatics,  Gda\'{n}sk University of Technology, Poland.}\\
Department of Computer Science\\
Purdue University\\
W. Lafayette, IN 47907\\
 U.S.A.\\
 spa{@}cs.purdue.edu}

\maketitle

\begin{abstract}
Tries are among the most versatile and widely used data structures on words.
In particular, they are used in  fundamental sorting algorithms such as
radix sort which we study in this paper.
While the performance of radix sort and tries
under a realistic probabilistic model for the generation of words is of
significant importance, its analysis, even for simplest memoryless sources,
has proved difficult.
In this paper we consider a more realistic model where words are generated by a Markov source.
By a novel use of the contraction method combined with
 moment transfer techniques we prove a central limit
theorem for the complexity of radix sort and for the external path length
in a trie.
This is the first application of the
contraction method to the analysis of algorithms and
data structures with Markovian inputs; it relies on the use of systems of stochastic recurrences combined with a product version of the Zolotarev metric.
\end{abstract}


\section{Introduction}


\emph{Tries} are prototype data structures useful for many indexing
and retrieval purposes.
Tries were first proposed by de-la-Briandais in 1959
\cite{de-la-Briandais59} for information processing.
Fredkin in 1960 suggested the current name,
part of the word re\emph{trie}val \cite{Knuth98,Mahmoud92,spa-book}.
They are pertinent to (internal) structure of (stored) words
and several splitting procedures used in diverse contexts
ranging from document taxonomy to IP addresses lookup, from data
compression to dynamic hashing, from partial-match queries to
speech recognition, from leader
election algorithms to distributed hashing tables and graph compression.

Tries are trees whose
nodes are vectors of characters or digits; they are a natural choice
of data structure when the input records involve the notion of
alphabets or digits.  Given a sequence of $n$
binary strings, we
construct a trie as follows. If $n=0$ then the trie is empty. If
$n=1$ then a single external node holding the word is allocated.
If $n\ge1$ then the trie consists of a root (i.e., internal) node
directing strings to two subtrees according to the first symbol
of each string, and strings directed to the same subtree recursively generate a trie among themselves, see Figure \ref{fig_trie_radix} and Section \ref{sec:2} for a more formal definition.
The \emph{internal nodes} in tries are branching nodes, used merely to
direct records to each subtrie; the record strings are all stored in
\emph{external nodes}, which are leaves of such tries.

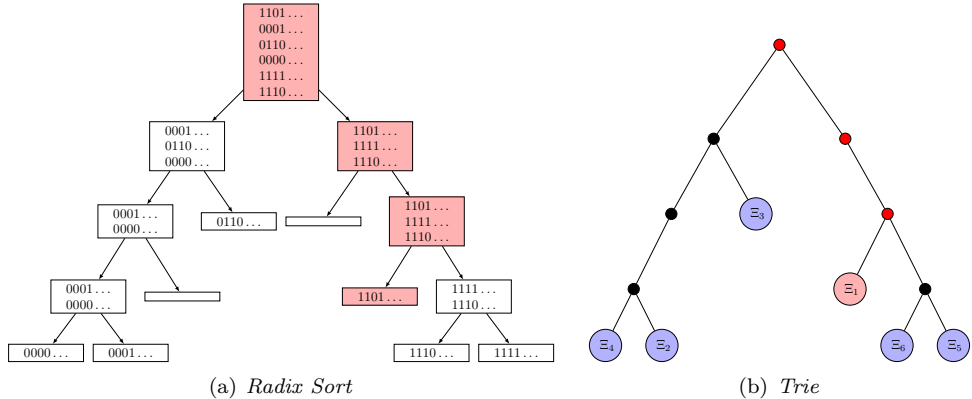
\begin{figure}
\captionsetup{singlelinecheck=off} 
\begin{center}

\subfigure[\emph{Radix Sort}]{
\scalebox{0.5}{
\begin{tikzpicture}[edge from parent/.style={draw,-latex}]
\tikzstyle{level 0}=[level distance=1cm]
\tikzstyle{level 1}=[sibling distance=50mm,  level distance=2.5cm]
\tikzstyle{level 2}=[sibling distance=27.5mm,  level distance=2cm]
\tikzstyle{level 3}=[sibling distance=25mm, level distance=2cm]
\tikzstyle{level 4}=[sibling distance=22.5mm, level distance=1.5cm]
\node[fill=red!30,draw,text width=5em, text centered]{$1101\ldots$ \\ $0001\ldots$ \\ $0110\ldots$ \\ $0000\ldots$ \\ $1111\ldots$ \\$1110\ldots$}
  child{node[draw,text width=5em, text centered]{$0001\ldots$ \\ $0110\ldots$ \\ $0000\ldots$ \\}
    child{node[draw,text width=5em, text centered]{$0001\ldots$ \\ $0000\ldots$}
      child{node[draw,text width=5em, text centered]{$0001\ldots$ \\$0000\ldots$}
	child{node[draw,text width=5em, text centered]{$0000\ldots$}
	}
	child{node[draw,text width=5em, text centered]{$0001\ldots$}
	}
      }
      child{node[draw,text width=5em, text centered]{\;}
      }
    }
    child{node[draw,text width=5em, text centered]{$0110\ldots$}
    }
  }
  child{node[fill=red!30,draw,text width=5em, text centered]{$1101\ldots$ \\ $1111\ldots$ \\$1110\ldots$}
    child{node[draw,text width=5em, text centered]{\;}
    }
    child{node[fill=red!30,draw,text width=5em, text centered]{$1101\ldots$ \\ $1111\ldots$ \\$1110\ldots$}
      child{node[fill=red!30,draw,text width=5em, text centered]{$1101\ldots$}
      }
      child{node[draw,text width=5em, text centered]{$1111\ldots$ \\$1110\ldots$}
	child{node[draw,text width=5em, text centered]{$1110\ldots$}
	}
	child{node[draw,text width=5em, text centered]{$1111\ldots$}
	}
      }
    }
  }
;
\end{tikzpicture}
}
}
\subfigure[\emph{Trie}]{
\scalebox{0.5}{
 \begin{tikzpicture}
\tikzstyle{level 1}=[sibling distance=35mm, level distance=2.5cm]
\tikzstyle{level 2}=[sibling distance=22.5mm, level distance=2cm]
\tikzstyle{level 3}=[sibling distance=20mm, level distance=2cm]
\tikzstyle{level 4}=[sibling distance=15mm, level distance=1.5cm]
\tikzstyle{level 5}=[sibling distance=10mm]
\node [fill=red!150,circle,draw,inner sep=3pt] (r){}
    child{node [fill=black!255,circle,draw,inner sep=3pt] (0){}
	child{node [fill=black!255,circle,draw,inner sep=3pt] (00){}
	    child{node [fill=black!255,circle,draw,inner sep=3pt] (000){}
		child{node [fill=blue!30,circle,draw,inner sep=3pt] (0000){ $\;\Xi_4$}
		}
		child{node [fill=blue!30,circle,draw,inner sep=3pt] (0001){ $\;\Xi_2$}
		}
	    }
	    child[fill=none] {edge from parent[draw=none]}
	}
	child{node [fill=blue!30,circle,draw,inner sep=3pt] (01) { $\;\Xi_3$}
	}
    }
    child{node [fill=red!150,circle,draw,inner sep=3pt] (1){}
	child[fill=none] {edge from parent[draw=none]}
	child{node [fill=red!150,circle,draw,inner sep=3pt] (11){}
	    child{node [fill=red!30,circle,draw,inner sep=3pt] (110) { $\;\Xi_1$}
	    }
	    child{node [fill=black!255,circle,draw,inner sep=3pt] (111){}
		child{node [fill=blue!30,circle,draw,inner sep=3pt] (1110) { $\;\Xi_6$}
		}
		child{node [fill=blue!30,circle,draw,inner sep=3pt] (1111) { $\;\Xi_5$}
		}
	    }
	}
    }
;

 \end{tikzpicture}
}
}
\end{center}
\caption[Radix sort and a trie.]{Radix sort and a trie applied to the 
strings: $\Xi_1=1101\ldots,\, \Xi_2=0001\ldots,\, \Xi_3=0110\ldots,\, 
\Xi_4=0000\ldots,\, \Xi_5=1111\ldots,\, \Xi_6=1110\ldots$
Note that radix sort places $\Xi_1$ into three sublists, also called buckets, (and has to read the first three symbols of $\Xi_1$)
whereas the node storing $\Xi_1$ has depth three in the corresponding trie.}
\label{fig_trie_radix}
\end{figure}

Tries can be used in many fundamental algorithms, in particular for sorting
known as radix sort or more precisely most significant digit radix sort
\cite{Knuth98}.  In this cases, the $n$ strings are binary 
representations of keys to be sorted.  They are inserted in a trie 
as described above. A so-called depth-first traversal of
the trie starting at the root node will visit each key in sorted order.
In other words, keys that start with a $0$ are moved to the left subtree
also called a left bucket, while the other keys are stored in the right
subtree or right bucket. In the sequel, we sort keys in the left and the
right buckets using the second symbol, an so on as shown in
Figure \ref{fig_trie_radix}(a).
A {\em  recursive} description of the radix sort algorithm
is presented in Section \ref{sec:2}. In this paper,
we shall use the trie and radix sort paradigms exchangeably.
The complexity of such radix sort is equal to the external path length of
the associated tries, that is, the sum of the lengths of the paths
from the root to all external nodes.


We study the limit law of the radix sort complexity and
the external path length of a trie
built over $n$ binary strings generated by a Markov source.
More precisely, we assume
that the input is a sequence of $n$ independent and identically
distributed random strings, each being composed of an infinite
sequence of symbols such that the next symbol depends on the previous
one and this dependence is governed by a given transition matrix
(i.e., Markov model).

Digital trees, in particular, tries have been intensively studied for the
last thirty years
\cite{clflva01,de84,jare88,jasz95,jaszta01,kiprsz89,de02,de05, Knuth98,Mahmoud92,spa-book},
mostly under Bernoulli (memoryless)
model assumption.
The typical depth under the Markov model was analyzed in \cite{jaszta01},
however, not the external path length. The external path length is more challenging due to
stronger dependency, see \cite{spa-book}. In fact, this is already observed for tries under the 
Bernoulli model \cite{spa-book}. In this paper we establish a central limit theorem for the
external path length in a trie built over a Markov model using a novel
use of the {\it contraction method}.

The contraction method
was introduced in 1991 by Uwe R\"osler \cite{Ro91} for the
distributional analysis of the complexity of the Quicksort algorithm. It was then developed independently by R\"osler and by Rachev and R\"uschendorf \cite{RaRu95} in the early 1990's.
Over the last 20 years this approach, which is based on exploiting an
underlying contracting map on a space of probability distributions,
has been developed as a fairly universal tool for the analysis of
recursive algorithms and data structures. Here, randomness may come
from a stochastic model for the input or from randomization within the
algorithms itself (randomized algorithms). General  developments of this
method were presented in
\cite{Ro92,RaRu95,Ro99,NeRu04,NeRu04b,DrJaNe08,JaNe08,NeSu12}
with numerous applications in computer science, information theory,
and networking.

The contraction method has been used in the analysis of tries and
other data structures only under the
symmetric Bernoulli model (unbiased memoryless source)
\cite[Section 5.3.2]{NeRu04},
where limit laws for the size and the external path length of tries were
re-derived. The application of the method there was heavily based on the fact
that precise expansions of the expectations were available, in particular
smoothness properties of periodic functions appearing in the linear terms
as well as  bounds on error terms which were $\bo(1)$ for the size and
$\bo(\log n)$ for the path lengths. It should be observed that
even in the asymmetric Bernoulli model such error terms seem to be out of
reach for classical analytic methods; see the discussion in
Flajolet, Roux, and Vall{\'e}e \cite{flrova10}. Hence,
for the more general Markov source model considered in the present paper
we develop a novel use of the contraction method.

Furthermore, the contraction method
applied to Markov sources hits another snag, namely,
the Markov model is not preserved when decomposing the trie at its root into
its left and right subtree. The initial distribution of the
Markov source is changed when looking at these subtrees.
To overcome these problems a couple of new ideas are used for
setting up the contraction method: First of all, we will use a system of
distributional recursive equations, one for each subtree.
We then apply the contraction method to this system of recurrences capturing the
 subtree processes and prove asymptotic 
normality for   the path lengths conditioned on the initial distribution.
In fact, our approach avoids dealing with multivariate recurrences and instead
we reduce the whole analysis to a system of  one-dimensional equations.
To come up with an appropriate contracting map we use a product version of the Zolotarev metric.

We also need asymptotic expansions of the mean and the variance for applying the
contraction method. In contrast to very precise information
on periodicities of linear terms for the symmetric Bernoulli model
mentioned above and in view of the results in \cite{flrova10} mentioned above we cannot expect to obtain similarly precise expansions. In fact, our convergence proof does only require the leading
order term together with a Lipschitz continuity
property for the error term.
The lack of a precise expansion is compensated by this Lipschitz continuity combined with a self-centering argument to obtain sufficiently tight  control on error terms.

For the derivation of such an expansions of the mean (and the variance) we use moment transfer theorems. Such theorems were largely developed by H.-K.~Hwang, see, e.g., \cite{hw03,fuhwza10,fuhwza14,bahw98}, for the control of moments related to one-dimensional recurrences. We extend such theorems to systems of recurrences as they occur for the analysis of our Markov model. For the expansion of the variance we also make use of a construction due to Schachinger \cite{Schach95}.

This is the first application of
the contraction method to the analysis of algorithms and data
structures with Markovian inputs. Our results  were announced in the extended abstract \cite{LeNeSz13}.
The methodology developed is general enough to
cover related quantities and structures as well.
Our approach also applies with  minor adjustments at least to the path lengths of digital search trees and PATRICIA tries under
the Markov source model, see the dissertation of the first mentioned author \cite{Leckey2015}.

The Markov source model is more realistic and more flexible than the (memoryless) Bernoulli model. Even more general models have been analyzed in the context of tries.  
Vall{\'e}e \cite{vallee2001} introduced the dynamical source models which, in particular, cover the Markov model. The analysis of dynamical sources for tries started
with the work of Cl{\'e}ment, Flajolet and Vall{\'e}e in \cite{clflva01}, including the asymptotic of the expectation of several trie parameters such as height,
size and the depth/external path length. There is a limit theorem for the depth in tries for special (so-called tame)  dynamical sources, see   \cite{CeVa15},
and a limit theorem for the depth in the (closely related) digital search tree for two types of general sources, see \cite{HuVa14}. However, a
limit theorem for the external path length in tries and the complexity of radix sort has not yet been derived for dynamical sources.
\\

\noindent
{\bf Notations:}
Throughout this paper we use the Bachmann--Landau symbols, in particular
the big $\bo$ notation. We declare $x\log x:=0$ for $x=0$.
By $B(n,p)$ with $n\in\N$ and $p\in [0,1]$ the binomial distribution
is denoted, by $B(p)$ the Bernoulli distribution with success
probability $p$, by ${\cal N}(0,\sigma^2)$ the centered normal
distribution with variance $\sigma^2>0$. We use $C$ as a generic
constant that may change from one occurrence to another.

\section{Main Results}\label{sec:2}

In this section we first describe succinctly the radix sort and his relation
to tries. Then we present our probabilistic model, and
the main result of this paper. \\

\noindent
{\bf Radix sort}. Given $n$ keys represented by binary strings,
we can sort them in the following way. We first split them according to
the first bit: those string starting with a $0$ go to the left bucket,
while the others to the right bucket. In each bucket we sort remaining strings
in the same manner using the second bit. And so on. At the end we
read all keys  from left to right and
all $n$ keys are sorted, see Figure \ref{fig_trie_radix}. 
This is called a radix sort \cite{Knuth98}.
 The number of inspected bits needed to sort such $n$ keys (strings)
 is denoted by $B_n$ and called it in short the number of bucket 
operations. It measures the complexity of radix sort.
We study its limiting distribution in this paper.

It is easy to see that we can achieve the same result by building a trie
from $n$ strings and visit all external nodes in a tree traversal.
Then $B_n$ can be interpreted as the length of the external path length,
that is, the sum of all paths from the root to all external nodes.\\

{\bf The Markov source:} We now define the probabilistic model
for string generation. We shall assume that binary data strings
over the alphabet $\Sigma=\{0,1\}$ are generated by
a homogeneous Markov source.  In general, a homogeneous Markov chain is
given by its initial distribution $\mu=\mu_0 \delta_0 + \mu_1 \delta_1$
on $\Sigma$ and the transition matrix $(p_{ij})_{i,j \in \Sigma}$.
Here, $\delta_x$ denotes the Dirac measure in $x\in \R$.
Hence, the initial state is $0$ with probability $\mu_0$ and $1$
with probability $\mu_1$. We have $\mu_0,\mu_1\in [0,1]$ and $\mu_0+\mu_1=1$.
A transition from state $i$ to $j$ happens with probability
$p_{ij}$, $i,j\in \Sigma$. Now, a data string is generated as the
sequence of states visited by the Markov chain.
In the Markov source model assumed subsequently
all data strings are independent
and  identically distributed according to the given Markov chain.

We always assume that $p_{ij}>0$ for all $i,j\in \Sigma$.
Hence, the Markov chain is ergodic and has a stationary distribution,
denoted by $\pi=\pi_0 \delta_0 + \pi_1 \delta_1$. We have
\begin{align} \label{stat_dist}
\pi_0=\frac{p_{10}}{p_{01}+p_{10}},\qquad \pi_1=\frac{p_{01}}{p_{01}+p_{10}}.
\end{align}
Note however, that our Markov source model does not require
the Markov chain to start in its stationary distribution.

The case $p_{ij}=1/2$ for all $i,j\in\Sigma$ is essentially the
symmetric Bernoulli model (only the first bit may have a different
initial distribution). The symmetric Bernoulli model has already
been studied thoroughly also with respect to the external path
length of tries;
see \cite{kiprsz89,NeRu04}.
Hence, we exclude this case subsequently.
For later reference, we summarize our  conditions as:
\begin{align} \label{cond_prob}
p_{ij}\in (0,1) \mbox{ for all } i,j\in\Sigma,
\qquad p_{ij}\neq \frac{1}{2} \mbox{ for some } (i,j)\in\Sigma^2.
\end{align}

The entropy rate of the Markov chain plays an important role in the asymptotic
behavior of the performance of radix sort.
In particular, it determines the leading order constant of
the average number of bucket operations (path length) performed by radix sort.
The entropy rate for our Markov chain is given by
\begin{eqnarray}\label{def_ent}
H:= -\sum_{i,j\in \Sigma} \pi_i\, p_{ij} \log p_{ij}=\sum_{i\in \Sigma} \pi_i H_i,
\end{eqnarray}
where $H_i:=-\sum_{j\in\Sigma} p_{ij} \log p_{ij}$ is the
entropy of a transition from state $i$ to the next state.
Thus, $H$ is obtained as weighted average of  the entropies of
all possible transitions  with weights according to the
stationary distribution $\pi$. \\

Our main result concerning the distribution of the number of bucket operations
in radix sort or the path length in a trie is presented next. We will write $B_n^\mu$ for $B_n$ to make its dependence on the initial distribution explicit.

\begin{thm}\label{main_thm}
The number $B_n^\mu$ of bucket operations under the Markov source model with conditions \eqref{cond_prob} satisfies, as $n\rightarrow\infty$,
$$\E\left[B_n^\mu\right]=\frac 1 H n \log n +\bo(n),\qquad \V\left(B_n^\mu\right)=\sigma^2 n \log n + \bo \left(n \sqrt{\log n}\right)$$
where the entropy rate $H$ is defined in (\ref{def_ent}) and $\sigma^2$ is given by
\begin{align*}
\sigma^2=\frac{\pi_0 p_{00}p_{01}}{H^3} \left( \log (p_{00}/p_{01})+ \frac {H_1-H_0}{p_{01}+p_{10}}\right)^2+\frac{\pi_1 p_{10}p_{11}}{H^3}\left( \log (p_{10}/p_{11})+ \frac {H_1-H_0}{p_{01}+p_{10}}\right)^2.
\end{align*}
Moreover, as $n\rightarrow\infty$,
\begin{align*}
 \frac{B_n^\mu-\E[B_n^\mu]}{\sqrt{\V(B_n^\mu)}}\stackrel{d}{\longrightarrow}\mathcal{N}(0,1)
\end{align*}
where $\mathcal{N}(0,1)$ denotes a random variable with the standard normal distribution.
\end{thm}
The analysis of $B_n^\mu$ is based on a system of recursive distributional
equations discussed in the next section. Section \ref{sec:transfer}
contains some moment-transfer theorems
that are used in the analysis of mean and variance.
These theorems are applied to the analysis of the mean in
section \ref{sec:mean} in order to derive the asymptotic expansion in
Theorem \ref{main_thm} as well as a more detailed study of the remaining
term $f_\mu(n):=\E[B_n^\mu]-n \log n /H$ which is necessary
to obtain the limit law in section \ref{sec:limit}.

The first order asymptotic of $\V(B_n^\mu)$
with uniform error term is derived in section
\ref{sec:var}. It is based on the moment-transfer theorems  from
section \ref{sec:transfer} but requires some additional ideas such
as a splitting of $B_n^\mu$ into a suitable sum and a poissonization argument.

Finally, the limit theorem is establish in section \ref{sec:limit}.
The proof is based on the contraction method.
In fact, the asymptotic analysis of the moments enables us to
apply this technique. It is possible to obtain a more detailed
asymptotic expansion of the mean by analytical techniques
however, without the analysis of the increments in proposition
\ref{thEwLip} the analysis in section \ref{sec:limit}
would require an asymptotic expansion
up to the order of $o(\sqrt{n\log n})$.
It should be pointed out that analytic techniques allows asymptotics
of the mean and the variance up to $o(n)$ \cite{spa-book}.

\section{Recursive Distributional Equations}\label{sec:rec_eq}

We formulate in this section a system of distributional recurrences to
capture the distribution of the number of bucket operations.
Our subsequent analysis is entirely based on these equations.
In the sequel, we phrase our discussion in terms of the radix sort algorithm.

We denote by $B_n^\mu$ the number of bucket operations
(i.e., number of bits inspected by radix sort) performed sorting $n$
data under the Markov source model with initial distribution $\mu$
using the radix sorting algorithm. We have $B_0^\mu=B_1^\mu=0$ for all
initial distributions $\mu$. The transition matrix is given in
advance and suppressed  in the notation. We abbreviate
$B_n^i:=B_n^{\rho_i}$ for $i\in\Sigma$ and
$\rho_i= p_{i0}\delta_0 +p_{i1}\delta_1$. We will study
$B_n^0$ and $B_n^1$. From the asymptotic behavior of these
two sequences we can then directly obtain corresponding results
for $B_n^\mu$ for an arbitrary initial distribution
$\mu=\mu_0 \delta_0 + \mu_1 \delta_1$ as follows:
We denote by $K_n$ the number of data among our $n$ that
start with bit $0$. Then $K_n$ has the binomial $B(n,\mu_0)$ distribution.
In the Markov source model the distribution of the second bit
of every data string that starts with bit $0$ is $\rho_0$.
In particular, for any data string $\Xi=\xi_1\xi_2\ldots$ in
the left bucket (i.e.~$\xi_1=0$) the remaining
suffix $\xi_2\xi_2\ldots$ is generated by a Markov source model with initial
distribution $\rho_0$ and the same transition matrix as the
original source. Similarly,
the remaining suffixes in the right bucket are generated by a
Markov source model with initial distribution $\rho_1$ and
the same transition matrix.
Moreover, by the independence of data strings within the Markov source model,
the number of bucket operations in the left bucket and the number
of bucket operations in the right bucket are independent conditionally on $K_n$.
This leads to the following stochastic recurrence:
\begin{eqnarray} \label{decomp_initial}
B_n^\mu \stackrel{d}{=} B_{K_n}^0 + B_{n-K_n}^1+n, \qquad n\ge 2,
\end{eqnarray}
where $(B_0^0,\ldots,B_n^0)$, $(B_0^1,\ldots,B_n^1)$ and
$K_n$ are independent and $\stackrel{d}{=}$ denotes that
left and right hand side have identical distributions.
We will see later that we can directly transfer asymptotic
results for $B_n^0$ and $B_n^1$ to general $B_n^\mu$ via
(\ref{decomp_initial}), see, e.g., the proof of Theorem \ref{thm:limit}.

In particular, (\ref{decomp_initial}) implies for $\mu=\rho_0$ that
\begin{eqnarray} \label{decomp_zero}
B_n^0 \stackrel{d}{=} B^0_{I_n} + B^1_{n-I_n}+n, \qquad n\ge 2,
\end{eqnarray}
with $(B_0^0,\ldots,B_n^0)$, $(B_0^1,\ldots,B_n^1)$ and $I_n$
independent binomially $B(n,p_{00})$ distributed.
A similar argument yields a recurrence for $B_n^1$.
Denoting by $J_n$ a binomially $B(n,p_{10})$ distributed
random variable, we have
\begin{eqnarray} \label{decomp_one}
B_n^1 \stackrel{d}{=} B^0_{J_n} + B^1_{n-J_n}+n, \qquad n\ge 2,
\end{eqnarray}
with $(B_0^0,\ldots,B_n^0)$, $(B_0^1,\ldots,B_n^1)$ and $J_n$ independent.
Our asymptotic analysis of $B_n^\mu$ is based on the distributional
recurrence system (\ref{decomp_zero})--(\ref{decomp_one})
as well as (\ref{decomp_initial}).

For further references, we abbreviate (\ref{decomp_zero}) and (\ref{decomp_one}) by
\begin{eqnarray}\label{decomp_zero_one}
B_n^i \stackrel{d}{=} B^0_{I_n^i} + B^1_{n-I_n^i}+n, \qquad n\ge 2,\, i\in\Sigma,
\end{eqnarray}
with $(B_0^0,\ldots,B_n^0)$, $(B_0^1,\ldots,B_n^1)$ and $I_n^i$ independent, $I_n^i$ binomial $B(n,p_{i0})$ distributed.

\section{Transfer Theorems for Mean and Variance}\label{sec:transfer}

Throughout this section, let $(a_i(n))_{n\in\N_0}$ and
$(\varepsilon_i(n))_{n\in\N_0}$ be real valued sequences for
$i\in\{0,1\}$. Furthermore, let $I_n^i$
follow the binomial distribution $B(n, p_{i0})$
for $i\in\{0,1\}$. Suppose that these sequences either satisfy
\begin{align}\label{rec_transfer_first_kind}
 a_i(n)&=\E [a_0(I_{n}^i)] + \E [a_{1}(n-I_{n}^i)] + \varepsilon_i(n),\qquad i\in\{0,1\},\, n\in\N,
\end{align}
which is the case for, e.g.,
$a_i(n)=\E[B_n^i]$
and $\varepsilon_i(n)=n \Ind_{[2,\infty)}(n)$, or satisfy
\begin{align}\label{rec_transfer_second_kind}
 a_i(n)&=p_{i0}\E [a_0(I_{n}^i)] + p_{i1}\E [a_{1}(n-I_{n}^i)] + \varepsilon_i(n),\qquad i\in\{0,1\},\, n\in\N,
\end{align}
which is the case for, e.g.,
$a_i(n)=f_i(n+1)-f_i(n)$ where
$f_i(n)=\E[B_n^i] -\frac{1}{H} n \log n$ and $\varepsilon_i(n)=1$.

Upper bounds on $\varepsilon_i(n)$ may be transferred to bounds on $a_i(n)$ by the following lemma:
\begin{lem}
 \label{lem_transf}
Assume that \eqref{rec_transfer_first_kind} holds. Then, $\varepsilon_i(n)=\bo(n^{\alpha})$ for an $\alpha\in\R$ and both $i\in\{0,1\}$ implies, as $n\rightarrow\infty$,
$$
 a_i(n)=\begin{cases} \; \bo(n),& \text{ if } \alpha<1,\\ \; \bo( n^\alpha), &\text{ if } \alpha>1,\\
\;\bo(n \log n ), & \text{ if } \alpha=1. \end{cases}
$$
\end{lem}
More precisely, the first order asymptotic of linear $\varepsilon_i(n)$ terms yield the following first order asymptotic of $a_i(n)$:
\begin{lem}\label{transf_lem_first_order}
Assume that \eqref{rec_transfer_first_kind} holds. Then, $\varepsilon_i(n)=c_i n +\bo(n^{\alpha})$ for $c_0,c_1\in\R$ and $\alpha<1$ and both $i\in\{0,1\}$
implies that, as $n\to\infty$,
$$a_i(n)= \frac{ \pi_0 c_0+\pi_1 c_1} H n \log n +\bo(n)$$
with constants $\pi_0,\pi_1$ and $H$ given in \eqref{stat_dist} and \eqref{def_ent}.
\end{lem}

Similarly, there are the following results on transfers for \eqref{rec_transfer_second_kind}:
\begin{lem}
\label{lem_transf_inc}
Assume that \eqref{rec_transfer_second_kind} holds. Then, $\varepsilon_i(n) =\bo( n^\alpha)$ for an $\alpha\in\R$ and both $i\in\{0,1\}$ implies that, as $n\rightarrow\infty$,
$$
 a_i(n)=\begin{cases} \; \bo(1)& \text{ if } \alpha<0,\\ \; \bo( n^\alpha) &\text{ if } \alpha>0,\\
\;\bo( \log n ) & \text{ if } \alpha=0. \end{cases}
$$
\end{lem}
\begin{lem}
\label{lem_transf_inc_first_order}
Assume that \eqref{rec_transfer_second_kind} holds. Then, $\varepsilon_i(n) =c_i+\bo( n^{-\alpha})$
for $c_i\in\R$, $\alpha>0$ and both $i\in\{0,1\}$ implies, as $n\rightarrow\infty$,
$$
 a_i(n)= \frac{ \pi_0 c_0 + \pi_1 c_1} H  \log n + \bo(1),\quad i\in\Sigma,
$$
with constants $\pi_0,\pi_1$ and $H$ given in \eqref{stat_dist} and \eqref{def_ent}.
\end{lem}
\begin{proof}[Proof of lemma \ref{lem_transf}]
 The proof relies on the fact that $I_n^0$ and $I_n^1$ are concentrated around their means $p_{00}n$ and $p_{10}n$.
This leads to a geometric decay in the size of the toll term when iterating \eqref{rec_transfer_first_kind} on the right hand side.
It is more convenient to work with the monotone sequences given by
\begin{align*}
C_i(n):= \sup\{|a_i(k)| : 0 \leq k\leq n \},\qquad C(n):=\max \{C_0(n), C_1(n)\},\quad i\in\Sigma, n\in\N_0.
\end{align*}
Due to the upper bound $|a_i(n)|\leq C(n)$ for both $i\in\{0,1\}$, an upper bound on $C(n)$ is sufficient to prove the assertion.
To this end, let $\max_{i,j\in\{0,1\}}\{p_{ij}\}<\delta<1$ be a constant (the exact value of $\delta$ does not matter) and decompose \eqref{rec_transfer_first_kind} into
\begin{align}\label{decomp_pf_trans_lem_1}
 |a_i(n)|&\leq \E[( C(I_n^i)+C(n-I_n^i)) \Ind_{\{I_n^i\in [(1-\delta)n, \delta n]\}}]+ C(n) \Prob(I_n^i\notin [(1-\delta)n, \delta n]) +|\varepsilon_i(n)|.
\end{align}
Note that at least one of the following three equalities needs to hold by definition:
 \begin{align*}
 C(n)=|a_0(n)|\quad \text{or}\quad C(n)=|a_1(n)|\quad \text{or} \quad C(n)=C(n-1).
\end{align*}
Thus, the assumption on $\varepsilon_i(n)$ implies that there exists a constant $L>0$ such that \emph{at least one} of the following
two bounds holds
\begin{equation}\label{pf_lem_asym_small_toll}
\begin{aligned}
\beta(n) C(n) &\leq \max_{i\in\Sigma}\left\{ \E[ (C(I_n^i)+C(n-I_n^i)) \Ind_{\{I_n^i\in [(1-\delta)n,\delta n]\}}]+Ln^\alpha\right\}\\
C(n)&\leq C(n-1),
\end{aligned}
\end{equation}
where $\beta(n):=1-2 \max_{i\in\Sigma} \{\Prob( I_n^i \notin [(1-\delta) n, \delta n])\}$ converges to $1$ by a Chernoff bound on the binomial distribution (or the central
limit theorem).
Now (\ref{pf_lem_asym_small_toll}) implies for any $\varepsilon>0$ by induction on $n$ that
$$C(n) \leq D n^{\max\left\{-\frac {\log 4 }{\log \delta }, 2\alpha\right\} }(1+\varepsilon)^n$$
where $D=D(\varepsilon)>0$ is a sufficiently large constant. This yields for any
$K>1$ the rough upper bound $C(n) = \bo( K^n)$.\\

\noindent
To refine this bound, note that a standard Chernoff bound on the binomial distribution implies the existence of a constant $c>0$ such that for all $n\geq 0$
$$|\beta(n)-1| \leq 4 e^{-c n}$$
which together with $ C(n)=\bo(K^n)$ for $1<K<e^c$ yields a constant $L^\prime>0$ such that
$$|\beta(n)-1| C(n) \leq L^\prime n^\alpha, \quad n\in\N.$$
Combined with \eqref{decomp_pf_trans_lem_1}, this bound implies by induction on $n$ that
$$C(n) \leq \tilde{L} n \sum_{j=0}^{\lfloor -\log n /\log \delta \rfloor} (\delta^{1-\alpha})^j$$
where $\tilde{L}=\max\left\{C(d+1), (L+L^\prime)\max\{\delta^{\alpha-1},1\}\right\}$.
Thus, the assertion holds by the asymptotic of the geometric sum.
\end{proof}
\begin{proof}[Proof of lemma \ref{transf_lem_first_order}] An easy calculation reveals that the sequences
 \begin{align*}
  \tilde{a}_i(n):= a_i(n) - \frac{ \pi_0 c_0+\pi_1 c_1} H n \log n
+\frac{c_{1-i}H_i} {(p_{10}+p_{01})H}n
,\quad n\in\N,\, i\in\{0,1\},
 \end{align*}
satisfy
\begin{align*}
\tilde{a}_i(n)&=\E [\tilde{a}_0(I_{n}^i)] + \E [\tilde{a}_{1}(n-I_{n}^i)] + \bo\left(n^{\max\{\alpha, 1/3\}}\right).
\end{align*}
Thus, lemma \ref{lem_transf} yields $\widetilde{a}_i(n)=\bo(n)$
and the assertion follows.
More precisely,
note that the transformed sequences satisfy for all $n\in\N$ and $i\in\{0,1\}$
\begin{align*}
 \tilde{a}_i(n)=\E[\tilde{a}_0(I_n^i)]+\E[\tilde{a}_1(n-I_n^i)]+\tilde{\varepsilon}_i(n)
\end{align*}
with, for $h(x):=x\log x$,
\begin{align*}
\tilde{\varepsilon}_i(n)&=\varepsilon_i(n)-c\left(h(n) -\E[h(I_n^i)+h(n-I_n^i)]\right) \\
&\quad+\frac{c_{1-i}H_i} {(p_{10}+p_{01})H}n-\frac{c_{1}H_0} {(p_{10}+p_{01})H}n p_{i0}
-\frac{c_{0}H_1} {(p_{10}+p_{01})H}n p_{i1}.
\end{align*}
Thus, it only remains to show $\tilde{\varepsilon}_i(n)=\bo\left(n^{\max\{\alpha, 1/3\}}\right)$.
To this end, note that
\begin{align*}
 &h(n) -\E[h(I_n^i)+h(n-I_n^i)]\\
&=-\E[n h(I_n^i/n)+ n h(1-I_n^i/n)]\\
&=H_i n -n\E[h(I_n^i/n)-h(p_{i0})+h(1-I_n^i/n)-h(p_{i1})]\\
&=H_i n +\bo\left(n^{1/3}\right)
\end{align*}
where the last equality holds by the concentration of the binomial
distribution and the asymptotic of $\log (1+x)$ as $x\rightarrow 0$
(note that $\log(I_n^i /n)-\log(p_{i0})=\log(1+(I_n^i-np_{i0})/(np_{i0}))$).
Details can be found in the appendix, equation \eqref{eee_app}.
Therefore, an easy calculation yields
$\tilde{\varepsilon}_i(n)=\bo\left(n^{\max\{\alpha, 1/3\}}\right)$
and the assertion follows.
\end{proof}
\begin{proof}[Proof of lemma \ref{lem_transf_inc}]
The idea is essentially the same as in the proof of lemma \ref{lem_transf}:
Once again,
it is more convenient to work with the monotone sequences $(C_i(n))_{n\geq 0}$ and $(C(n))_{n\geq 0}$
 given by
$$C_i(n):=\sup\{|a_i(k)| : 0 \leq k\leq n \},\qquad C(n):=\max\{C_0(n), C_1(n)\}, \quad n\in\N_0, i\in\Sigma.$$
With $\max_{i,j\in\{0,1\}}\{p_{ij}\}<\delta < 1$ equation \eqref{rec_transfer_second_kind} may be decomposed into
\begin{align*}
 |a_i(n)|&\leq  \E[(p_{i0} C_0(I_n^i)+p_{i1} C_1(n-I_n^i)) \Ind_{\{I_n^i\in [(1-\delta)n, \delta n]\}}]+C(n)\Prob(I_n^i\notin [(1-\delta)n, \delta n])+ |\varepsilon_i(n)|
\end{align*}
As in the proof of \ref{lem_transf} this implies $C(n)=\bo(K^n)$ for any constant $K>1$ and, by a standard Chernoff bound on the binomial distribution
\begin{align*}
 |a_i(n)| \leq \E[(p_{i0} C_0(I_n^i)+p_{i1} C_1(n-I_n^i)) \Ind_{\{I_n^i\in [(1-\delta)n, \delta n]\}}] + \bo(n^\alpha).
\end{align*}
One obtains by induction on $n$ that
$$C(n) \leq \tilde{L} \sum_{k=0}^{\lfloor - \log n / \log \delta \rfloor} \delta^{-\alpha j}$$
and the assertion follows by the asymptotic behavior of the geometric sum.
\end{proof}
\begin{proof}[Proof of lemma \ref{lem_transf_inc_first_order}] An easy calculation reveals that the sequences
\begin{align*}
 \tilde{a}_i(n)&:= a_i(n)- L g(n) +\frac{c_{1-i} H_i} {(p_{01}+p_{10}) H},\quad i\in\{0,1\},\,n\in\N
\end{align*}
with $L=( \pi_0 c_0 + \pi_1 c_1) / H$ satisfy
\begin{align*}
  \tilde{a}_i(n)&=p_{i0}\E [\tilde{a}_0(I_{n}^i)] + p_{i1}\E [\tilde{a}_{1}(n-I_{n}^i)] +\bo\left( n^{-\min\{\alpha,1/2\}}\right).
\end{align*}
Thus, lemma \ref{lem_transf_inc} implies the assertion.
\end{proof}

\section{Analysis of the Mean} \label{sec:mean}
First we study the asymptotic behavior of the expected number of Bucket operations
with a precise error term needed to derive a limit law in Section \ref{sec:limit}.
\begin{thm}\label{thm:mean}
For the number $B_n^\mu$ of Bucket operations under the Markov source model with conditions (\ref{cond_prob}) we have
\begin{align*}
 \E[B_n^\mu] =\frac{1}{H} n\log n + \bo(n), \qquad (n\to\infty),
\end{align*}
with the entropy rate $H$ of the Markov chain given in (\ref{def_ent}). The $\bo(n)$ error term is uniform in the initial distribution $\mu$.
\end{thm}
Our proof of Theorem \ref{thm:mean} as well as the corresponding limit law in Theorem \ref{thm:limit} depend on refined
properties of the $\bo(n)$ error term that are first obtained for the initial distributions
$\rho_0=p_{00}\delta_0+p_{01}\delta_1$ and $\rho_1=p_{10}\delta_0+p_{11}\delta_1$ and then generalized to arbitrary initial
distribution via (\ref{decomp_initial}). For those initial distributions we
 denote the error term for all $n\in\N_0$ and $i\in \Sigma$ by
\begin{align}\label{def_fi}
 f_i(n) := \E[B_n^i]-\frac{1}{H} n\log n.
\end{align}
The following Lipschitz continuity of $f_0$ and $f_1$ is crucial for our further analysis:
\begin{prop}\label{thEwLip}
There exists a constant $C>0$ such that for both $i\in\Sigma$ and all $m,n\in\N_0$
\begin{align*}
|f_i(m)-f_i(n)| \leq C |m-n|.
\end{align*}
\end{prop}
In order to prove the Lipschitz continuity of the error terms $f_0$ and $f_1$ (proposition \ref{thEwLip}) we will analyze the increments of $(f_0(n))_{n\geq 0}$
and $(f_1(n))_{n\geq 0}$ and apply Lemma \ref{lem_transf_inc_first_order}. We use the following notation for the increments:

For a sequence $x=(x(n))_{n\ge 0}$ in $\R$ we denote its (finite forward) difference sequence by
$(\Delta x(n))_{n\ge 0}$, where
\begin{align*}
\Delta x(n) := (\Delta x)(n) :=  x(n+1)-x(n), \qquad n\in \N.
\end{align*}
Note that the order of operation is first applying the $\Delta$-operator to the sequence then evaluating the difference sequence at $n$. In particular, for any sequence  $(m_n)_{n\in\N}$ in $\N_0$ we have
\begin{align*}
\Delta x(m_n) = x(m_n+1)-x(m_n), \qquad n\in \N_0
\end{align*}
(and in general $\Delta x(m_n) \neq x(m_{n+1})-x(m_n)$).

In the analysis of $(\Delta f_i(n))_{n\geq 0}$, $i\in\Sigma$ we use the following
 Lemma which is a special case of Lemma 2 in Schachinger \cite{Schach95}.
\begin{lem}
\label{lemdelta_app}
For any real sequence $(a(n))_{n \geq 0}$ and binomially $B(n,p)$ distributed $X_n$ with $p\in(0,1)$ we have
\begin{align*}
\Delta \E[a(X_n)]= p \E[\Delta a(X_n)], \quad n\in \N.
\end{align*}
\end{lem}
\begin{proof}
Note that $X_{n+1}\stackrel{d}{=} X_n + B$ in which $B$ and $X_n$ are independent and $\Prob(B=1)=p=1-\Prob(B=0)$. This yields
\begin{align*}
 \Delta \E[a(X_n)]= \E[a(X_n+B)-a(X_n)]= p\E[\Delta a(X_n)]
\end{align*}
which is the assertion.
\end{proof}
\begin{proof}[Proof of proposition \ref{thEwLip}]
 Note that \eqref{decomp_zero_one} implies
\begin{align*}
 f_i(n) = \E[f_0(I_n^i)]+\E[f_1(n-I_n^i)] +\varepsilon_i(n)
\end{align*}
with the toll function
\begin{align*}
 \varepsilon_i(n)&=n- \frac 1 H (n \log n -\E[I_n^i \log I_n^i] -\E[(n-I_n^i) \log (n-I_n^i)].
\end{align*}
Thus, lemma \ref{lemdelta_app} yields for the increments $a_i(n):=\Delta f_i(n)$
\begin{align*}
 a_i(n)= p_{i0}\E[a_0(I_n^i)]+p_{i1} \E[a_1(n-I_n^i)] + \Delta\varepsilon_i(n).
\end{align*}
Moreover, another application of lemma \ref{lemdelta_app} yields
\begin{align*}
  \Delta\varepsilon_i(n)=1 - \frac 1 H ( \Delta h(n) -p_{i0}\E[\Delta h(I_n^i)]- p_{i1}\E[\Delta h(n-I_n^i)])
\end{align*}
where $h(x):= x \log x$. Since $\Delta h(n)= \log(n+1)+n \log (1+1/n)=\log(n+1) +1 +\bo(1/n)$, one obtains
\begin{align*}
 \Delta\varepsilon_i(n)&= 1- \frac 1 H ( \log(n+1) -p_{i0}\E[\log(I_n^i+1)]-p_{i1}\E[\log(n-I_n^i +1)]+\bo(1/n)\\
&=1- \frac 1 H (-p_{i0}\log p_{i0} -p_{i1}\log p_{i1}) +\bo (n^{-1/2}).
\end{align*}
The last equation is based on the fact that
$\E[\log((I_n^i+1)/(n+1))]=\log(p_{i0})+\bo(n^{-1/2})$
for any binomially $B(n,p_{i0})$ distributed $I_n^i$ (details are given in the appendix, equation \eqref{aaa_app}).
Therefore, lemma \ref{lem_transf_inc_first_order} implies
$\Delta f_i(n) = L \log n +\bo(1)$ with a constant
\begin{align*}
L=\frac 1 H \left( \pi_0 \left(1- \frac 1 H (-p_{00}\log p_{00}
-p_{01}\log p_{01})\right)+\pi_1\left(1- \frac 1 H
(-p_{10}\log p_{10} -p_{11}\log p_{11})\right)\right)
=0.
\end{align*}
Thus, $\Delta f_i(n)$ is bounded and the assertion follows.
\end{proof}
\begin{proof}[Proof of theorem \ref{thm:mean}]
For $\mu=p_{i0}\delta_0+p_{i1}\delta_1$, $i\in\{0,1\}$ theorem \ref{thm:mean} is an immediate consequence of proposition \ref{thEwLip}.
For the general case let $\nu_i(n):= \E[B_n^i]$. Then, the distributional recursion \ref{decomp_initial} yields
\begin{align*}
 \E[B_n^\mu]=\E[\nu_0(K_n)]+\E[\nu_1(n-K_n)] + n.
\end{align*}
Thus, $\nu_i(n)= n \log n / H +\bo(n)$ implies
\begin{align*}
 \E[B_n^\mu]= \frac 1 H n \log n + \frac n H \E[ h(I_n^i/n)]+\E[h(1-I_n^i/n)] + \bo(n)
\end{align*}
where $h(x)=x\log x$. Since $h$ is uniformly bounded on $(0,1]$, the assertion follows.
\end{proof}

\section{Analysis of the Variance}\label{sec:var}

In this section we establish precise growth of the variance with a uniform bound.
We prove the following theorem.

\begin{thm} \label{thm:variance}
For the number $B_n^\mu$ of Bucket operations under the
Markov source model with conditions (\ref{cond_prob}) we have, as $n\to\infty$,
\begin{align}\label{var_asympt}
\V(B_n^\mu) =\sigma^2 n\log n + \bo\left(n \sqrt{\log n}\right),
\end{align}
where $\sigma^2>0$ is independent of the initial distribution $\mu$ and
given by
\begin{align}
 \label{var_formula}
\sigma^2=\frac{\pi_0 p_{00}p_{01}}{H^3} \left( \log (p_{00}/p_{01})+ \frac {H_1-H_0}{p_{01}+p_{10}}\right)^2+\frac{\pi_1 p_{10}p_{11}}{H^3}\left( \log (p_{10}/p_{11})+ \frac {H_1-H_0}{p_{01}+p_{10}}\right)^2.\end{align}
\end{thm}
In order to derive the first order asymptotics of the variance without studying the mean in
detail, we extend an idea of Schachinger in \cite{Schach95} to Markov Sources. The main
ingredient is to split the number of Bucket operations into a sum of two random variables in which mean and variance
of the first random variable is easy to derive and the variance of the second random variable is small (i.e.~$\bo(n)$).

Once again, for $i\in\Sigma$ and $n\in\N_0$ let $I_n^i$ be a Binomial $B(n,p_{i0})$ distributed random variable.
Now let $(X_n^0,Z_n^0)_{n\in\N_0}$, $(X_n^1, Z_n^1)_{n\in\N_0}$ and $(I_n^0,I_n^1)_{n\in\N_0}$ be independent
sequences of random variables with finite second moments that satisfy the initial conditions
$$ X_n^i=Z_n^i=0, \quad i\in\Sigma, n\leq 1$$
and, for all $n\geq 2$ and $i\in\Sigma$
\begin{equation}
\label{split_for_var}
\begin{pmatrix}  X_{n}^{i} \\ Z_{n}^{i} \end{pmatrix} \stackrel{d}{=}
\begin{pmatrix}  X_{I_n^i}^{0} \\ Z_{I_n^i}^{0} \end{pmatrix}+\begin{pmatrix}  X_{n-I_n^i}^{1} \\ Z_{n-I_n^i}^{1} \end{pmatrix}
+\begin{pmatrix}  \eta_n^{i,1} \\ \eta_n^{i,2} \end{pmatrix},
\end{equation}
where the toll terms are given by $\eta_n^{i,1}=\eta_n^{i,2}=0$ for $n\leq 1$ and
\begin{equation} \label{def_etani_split_toll}
\begin{aligned}
\eta_n^{i,1} & := \frac 1 H \left(n\log(n)-\E \left[I_n^i \log \left(  {I_n^i} \right) + (n-I_n^i) \log \left (  {n-I_n^i} \right) \right]\right) \\
& \quad + \pi_{1-i} \frac {H_{1-i}-H_i} H n + \frac{H_1-H_0}{(p_{01}+p_{10})H} p_{i0}p_{i1}^{n-1} n,\qquad n\geq 2,\\
\eta_n^{i,2}&:=n-\eta_n^{i,1}.
\end{aligned}
\end{equation}
Since we have $\eta_n^{i,1}+\eta_n^{i,2}=n$, note that the sum $S_n^i:=X_n^i+Z_n^i$
satisfies the same initial conditions and the same stochastic recurrence as $B_n^i$, i.e.~equation
(\ref{decomp_zero_one}) and $S_n^i=0=B_n^i$ for $n\leq 1$. In particular, this implies that $S_n^i$ and $B_n^i$
have the same mean and variance. A discussion on the existence of a splitting satisfying (\ref{split_for_var})
and the equality of the moments of $S_n^i$ and $B_n^i$ is given in section \ref{sec:exsplitting}.\\

\noindent
{\bf Remark.} The choice of $\eta_n^{i,1}$ is motivated as follows: Since $Z_n^i$ should be small ($\E[Z_n^i]=\bo(n)$, $\V(Z_n^i)=\bo(n)$), $X_n^i$ should satisfy
$\E[X_n^i]\sim \frac 1 H n\log (n)$ which is the reason for the choice of the first summand in (\ref{def_etani_split_toll}).
The linear term is chosen to obtain $\eta_n^{i,1}\sim n$ and therefore $\eta_n^{i,2}=o(n)$ which implies a small
variance for $Z_n^i$. The last summand is chosen for some technical reasons to compensate the second one in the calculation of $\E[X_n^i]$.\\

The proof of theorem \ref{thm:variance} works as follows: first we
study the asymptotics of $\V(X_n^i)$ and $\V(Z_n^i)$ and then deduce the asymptotics of $\V(B_n^i)$ by the
following Lemma:
\begin{lem}
 \label{lem_var_sum}
For any random variables $X,Y$ with finite second moments we have
\begin{align}
 \label{var_sum_bounds}
\left(\sqrt{\V(X)}-\sqrt{\V (Y)}\right)^2 \leq \V (X+Y) \leq \left(\sqrt{\V (X)}+ \sqrt{\V(Y)}\right)^2.
\end{align}
In particular, if sequences $(X_n)_{n\geq 0}, (Y_n)_{n\geq 0}$ with finite second moments
satisfy $\V (Y_n) =o(\V(X_n))$ then we have
\begin{align}
 \label{var_sum_asymp}
\V(X_n+Y_n) = \V(X_n) + \bo \left( \sqrt{\V (X_n)\V(Y_n)}\right).
\end{align}
\end{lem}
\begin{proof}
By the Cauchy-Schwarz inequality we have
$$| \CV (X,Y) | \leq \sqrt{\V (X)} \sqrt{\V (Y)}$$
which together with $\V(X+Y)= \V(X)+\V(Y) + 2 \CV (X,Y)$ implies (\ref{var_sum_bounds}).
Moreover, (\ref{var_sum_bounds}) obviously implies (\ref{var_sum_asymp}).
\end{proof}
The analysis of $\V(X_n^i)$ is done with lemma \ref{transf_lem_first_order}. This requires a detailed asymptotic expansion of $\E[X_n^i]$.
The choice of $\eta_n^{i,1}$ leads to the following representation of the mean:
\begin{lem}
 \label{lem_mean_changed_toll} Let $(X_n^i)_{n\in\N_0, i\in\Sigma}$ be as in (\ref{split_for_var}). Then we have for all $n\in\N_0$
\begin{align}
\label{mean_changed_toll}
 \E[X_n^0]= \frac 1 H n\log n + \frac {H_1-H_0}{(p_{01}+p_{10})H} n \textbf{1}_{\{n\geq 2\}}, \qquad \E[X_n^1]=\frac 1 H n \log n.
\end{align}
\end{lem}
\begin{proof}
 Let $\nu_X^i:\N_0\rightarrow\R$ be given by $\nu_X^i(n):=\E[X_n^i]$, $i\in\{0,1\}$. Note that $\nu_X^i$ is uniquely determined by its initial conditions $\nu_X^i(n)=0$ for $n\leq 1$ and
the recursion
\begin{align*}
 \nu_X^i(n)=\E[\nu_X^0(I_n^i)]+\E[\nu_X^1(n-I_n^i)]+\eta_n^{i,1},\quad i\in\Sigma,\, n\geq 2,
\end{align*}
which arises from the recursion \eqref{split_for_var}. Thus, it only remains to check that the choice given in \eqref{mean_changed_toll}
satisfies these conditions which is an easy calculation. Details are left to the reader.
\end{proof}
These expressions and lemma \ref{transf_lem_first_order} lead to the following asymptotics of $\V(X_n^i)$:
\begin{lem}\label{lem_var_changed_toll}
We have for both $i\in\Sigma$ as $n\rightarrow\infty$
$$\V \left(X_n^i\right)= \sigma^2 n \log n + \bo(n)$$
where $\sigma^2$ is given by \eqref{var_formula}.
\end{lem}
\begin{proof}
 Let $V_X^i(n):=\V(X_n^i)$ and $\nu_X^i(n):=\E[X_n^i]$ as in the previous proof. Then, the recursion \eqref{split_for_var} and the independence therein imply
\begin{align}\label{split_var_changed_toll}
 V_X^i(n)= \E[V_X^0(I_n^i)]+\E[V_X^1(n-I_n^i)] + \V(\nu_X^0(I_n^i)+\nu_X^1(n-I_n^i)).
\end{align}
It suffices to derive the first order asymptotic of
$\V(\nu_X^0(I_n^i)+\nu_X^1(n-I_n^i))$ to apply lemma \ref{transf_lem_first_order}.
To this end, note that by lemma \ref{lem_mean_changed_toll} with the notation $h(x):=x\log x$
\begin{align}\label{pf_var_changed_toll}
 \V(\nu_X^0(I_n^i)+\nu_X^1(n-I_n^i))= \V\left( \frac {h(I_n^i)+h(n-I_n^i)} H +\frac {H_1-H_0}{(p_{01}+p_{10})H} I_n^i +R_n^i\right)
\end{align}
where $R_n^i=-\frac {H_1-H_0}{(p_{01}+p_{10})H} \Ind_{\{I_n^i=1\}}$
and thus, $\V(R_n^i)=o(1)$.
Subtracting $n\log n$ in the variance on the right hand side of
\eqref{pf_var_changed_toll} yields
\begin{align*}
 \V(\nu_X^0(I_n^i)+\nu_X^1(n-I_n^i))=\V\left( \frac 1 H (I_n^i\log p_{i0}+(n-I_n^i)\log p_{i1})+\frac {H_1-H_0}{(p_{01}+p_{10})H} I_n^i+\widetilde{R}_n^i\right)
\end{align*}
where $\widetilde{R}_n^i=R_n^i+ \frac 1 H (I_n^i (\log(I_n^i/n) -\log p_{i0})+(n-I_n^i) (\log(1-I_n^i/n)-\log(p_{i1})))$. It is not hard to check that
$\V(\widetilde{R}_n^i)=\bo(\log n)$, as formally proved below.
Therefore, combined with lemma \ref{lem_var_sum} and  $\V(I_n^i)=p_{i0}p_{i1}n$
\begin{align*}
\V(\nu_X^0(I_n^i)+\nu_X^1(n-I_n^i))=\left(\frac 1 H \left(\log p_{i0}-\log p_{i1}+\frac {H_1-H_0}{(p_{01}+p_{10})}\right)\right)^2 p_{i0}p_{i1} n +\bo(n^{2/3}).
\end{align*}
Hence, the assertion follows by \eqref{split_var_changed_toll} and
lemma \ref{transf_lem_first_order}.\\

\noindent
To complete the proof we now establish that
$\V(\widetilde{R}_n^i)=O(\log n)$. Note that the function
\begin{align*}
 \phi:[0,1]\rightarrow\R,\quad x \to x(\log x -\log p_{i0}) + (1-x) (\log(1-x) -\log(1-p_{i0}))
\end{align*}
is bounded and that the derivative is given by  $\phi^\prime(x)= \log(x/p_{i0})- \log((1-x)/(1-p_{i0}))$. In particular,
there exists a constant $C>0$ such that for all sufficiently large $n$
\begin{align*}
 |\phi^\prime(x)| \leq C \sqrt{\frac {\log n } n },\qquad x\in\left[p_{i0}- \sqrt{(\log n ) /n}\,,\,p_{i0}+\sqrt{(\log n)/n}\right].
\end{align*}
One obtains
$$
\V(\phi(I_n^i/n) \Ind_{\{|I_n^i-n p_{i0}|\geq \sqrt{n \log n}\}})=\bo (n^{-2})
$$
by the boundedness of $\phi$ and a standard Chernoff bound and, by the previous
observations, the mean value theorem and a self centering argument
(let $J_n^i$ be an independent copy of $I_n^i$)
\begin{align*}
 &\V\left(\phi(I_n^i/n) \Ind_{\{|I_n^i-np_{i0}| < \sqrt{ n \log n}\}}\right)\\
&= \frac 1 2 \E\left[ \left( \phi(I_n^i/n) \Ind_{\{I_n^i-np_{i0}| < \sqrt{ n \log n}\}}
-\phi(J_n^i/n) \Ind_{\{|J_n^i-np_{i0}| < \sqrt{ n \log n}\}}\right)^2\right]\\
&=\frac  {C^2} 2 \frac {\log n } n \E\left[\left( I_n^i/n -J_n^i/n\right)^2\right]+\bo\left(n^{-2}\right)=\bo\left(\frac {\log n}{n^2}\right).
\end{align*}
The bound on $\V(\widetilde{R}_n^i)$ follows by lemma \ref{lem_var_sum} since $\widetilde{R}_n^i=R_n^i + n \phi(I_n^i/n)$ and $\V(R_n^i)=o(1)$.
\end{proof}
In order to derive the asymptotics of $\V(Z_n^i)$ we start with an upper bound on $\eta_n^{i,2}$:
\begin{lem}
 \label{lem_asym_etani}
For $\eta_n^{i,2}$ defined in (\ref{def_etani_split_toll}) we have for both $i\in\Sigma$, as $n\rightarrow\infty$
$$\eta_n^{i,2}=\bo\left(\log n \right).$$
\end{lem}
\begin{proof}
By the definition of $\eta_n^{i,2}$ in \eqref{def_etani_split_toll} one only needs to compute the asymptotic of
\begin{align*}
 h(n)-\E[h(I_n^i)]-\E[h(n-I_n^i)],\qquad h(n):=n \log n.
\end{align*}
Since $h(n)=\E[I_n^i \log n]+\E[(n-I_n^i)\log n]$, one obtains
\begin{align*}
  h(n)-\E[h(I_n^i)]-\E[h(n-I_n^i)]&=- n(\E[h(I_n^i/n)]+\E[h(1-I_n^i/n)])=n H_i - n \E[\phi(I_n^i/n)]
\end{align*}
where $H_i=-p_{i0}\log p_{i0}-p_{i1}\log p_{i1}$ and $\phi(x)= x(\log x - \log p_{i0})+(1-x)(\log(1-x)-\log(1-p_{i0})$. With the same arguments as at
the end of the previous proof one obtains $|\phi(x)|=\bo((\log n) /n)$ uniformly for $x\in[p_{i0}-\sqrt{(\log n) /n},p_{i0}+\sqrt{(\log n)/n}$ which implies
by a standard Chernoff bound on the binomial distribution that $n\E[\phi(I_n^i/n)]=\bo(\log n)$. Hence, $\eta_n^{i,1}=n + \bo(\log n)$ since
$H=\pi_0 H_0 +\pi_1 H_1$ and the assertion follows since $\eta_n^{i,2}=n-\eta_n^{i,1}$.
\end{proof}

Note that we have the following Lipschitz-continuity of the means:

\begin{lem}\label{lem_lip_small}
 For $i\in\Sigma$ let $\nu_Z^i : \N_0\rightarrow \R$ be given by
$$\nu_Z^i(n)=\E[Z_n^i],$$
where $(Z_n^i)_{n\in\N_0, i\in\Sigma}$ satisfies (\ref{split_for_var}).
Then, the functions $\nu_Z^0$ and $\nu_Z^1$ are Lipschitz continuous,
i.e.~there exists a constant $C>0$
such that for $i\in\Sigma$ and $n,m\in\N_0$ we have
$$|\nu_Z^i(n)-\nu_Z^i(m)| \leq C|n-m|.$$
\end{lem}
\begin{proof}
 Since we have
$$\E[X_n^i+Z_n^i]=\E[B_n^i]$$
the assertion immediately follows from proposition \ref{thEwLip} and lemma \ref{lem_mean_changed_toll}.
\end{proof}

The next step is to show that $\V(Z_i)=O(n)$ which we present in
lemma~\ref{lem_var_small} below. However, to establish it we need
another key ingredient, namely poissonization. In poissonization
one replaces $n$ by a Poisson $\Pi(\lambda)$ distributed random
variable $N$ to derive asymptotics as $\lambda\rightarrow \infty$.
This turns out to be easier than the original problem owing to some nice
properties of the Poisson process such as independence of
the splitting  processes.
The transfer lemma used after
poissonization is the following:
\begin{lem}
 \label{lem_asym_pois_var}
For $i\in\Sigma$ let $f_i:\R^+ \rightarrow \R$ be some function that is bounded on $(0,a]$ for all $a>0$.
Assume that there exist constants $p_0,p_1\in(0,1)$ such that for all $x>0$ and $i\in\Sigma$
\begin{align}
 \label{rec_func_small_toll}
f_i(x) = f_i(x p_{i}) + f_{1-i} (x (1-p_{i})) + \eta_i(x)
\end{align}
where $\eta_i:\R^+\rightarrow \R$ is some function.

Then, as $x\rightarrow\infty$, $\eta_i(x) = \bo( x^{1-\alpha})$ for some $\alpha>0$ and both $i\in\Sigma$ implies
$$f_i(x) =\bo (x),\quad i\in\Sigma.$$
\end{lem}
\begin{proof}
Iterating (\ref{rec_func_small_toll}), by
induction on $n$ we find that for a sufficiently large constant $C>0$
and all $n\in\N$
\begin{align*}
|f_i(x)| \leq C x \sum_{j=0}^{\left\lfloor - \frac{\log x} {\log p_\vee} \right\rfloor} p_\vee^{\alpha j},\qquad x\in [1, p_\vee^{-n}],\, i\in\{0,1\},
\end{align*}
where $p_\vee:=\max\{p_0,p_1,1-p_0,1-p_1\}$. The assertion follows since the sum converges as $x\rightarrow\infty$. Details on the induction are left to the reader.
\end{proof}
The crucial part after poissonization is to transfer the asymptotics as $\lambda\rightarrow\infty$ into asymptotics of the original problem. One way of doing this is the next lemma:
\begin{lem}\label{depois_lem}
Let $(a(n))_{n\in\N_0}$ be a real valued sequence. Moreover, let $N_\lambda$ be Poisson distributed with mean $\lambda>0$.
Then, as $n\rightarrow \infty$, $\Delta a(n):=a(n+1)-a(n)=\bo(\sqrt{n})$ implies
$$|a(n)-\E[a(N_n)]| =\bo\left( n \right).$$
\end{lem}
\begin{proof}
First note that $\Delta a(n)=\bo(\sqrt n)$ implies that there exists a constant $C>0$ such
that for all $n,m\in\N_0$
$$|a(n)-a(m)| =\left| \sum_{i=m\wedge n}^{m\vee n-1} \Delta a(i) \right|\leq C \sqrt{n+m} |n-m|.$$
Hence, we have that
\begin{align*}
|a(n)-\E[a(N_n)]|&\leq \E[|a(n)-a(N_n)|]
\leq C \E[ \sqrt{n+N_n} |N_n-n|]
\end{align*}
which implies the assertion by the Cauchy-Schwarz inequality.
\end{proof}

This finally leads to the following bounds on $\V(Z_n^0)$ and $\V(Z_n^1)$
which we present next.

\begin{lem}
\label{lem_var_small}
We have for both $i\in\Sigma$, as $n\rightarrow \infty$
$$\V(Z_n^i) = \bo( n ).$$
\end{lem}
\begin{proof} Let $V_Z^i(n):=\V(Z_n^i)$ and $\nu_Z^i(n):=\E[Z_n^i]$.
 First note that similar arguments to the ones given in the proof of lemma \ref{lem_var_changed_toll} reveal that
\begin{align}\label{var_rec_small}
 V_Z^i(n)=\E[V_Z^0(I_n^i)]+\E[V_Z^1(n-I_n^i)]+\V(\nu_Z^0(I_n^i)+\nu_Z^1(n-I_n^i)).
\end{align}
Since $\nu_Z^0$ and $\nu_Z^1$ are Lipschitz-continuous, we have $\V(\nu_Z^0(I_n^i)+\nu_Z^1(n-I_n^i))=\bo(n)$ which can be proven by a self centering argument
similar to the one at the end of the proof of lemma \ref{lem_var_changed_toll}. Thus, lemma \ref{lem_transf} yields the rough upper bound
\begin{align}\label{var_small_toll_rough_upper}
 \V(Z_n^i)=\bo(n\log n).
\end{align}
In order to refine this bound, let $N_\lambda$ be a Poisson distributed random variable with mean $\lambda>0$ which is independent
of $\{Z_n^i, I_n^i :  n\geq 0,i\in\{0,1\}\}$.
Then, (\ref{split_for_var}) implies for both $i\in\Sigma$
\begin{align}\label{pois_rec_small}
 Z_{N_\lambda}^i\stackrel{d}{=} Z_{N_{\lambda p_{i0}}}^0+ Z_{M_{\lambda p_{i1}}}^{1} + \eta_{N_\lambda}^{i,2}
\end{align}
where $N_{\lambda p_{i0}}:=I_{N_\lambda}^i$ and $M_{\lambda p_{i1}}:= N_\lambda- I_{N_\lambda}^i$. It is a well known fact, e.g.~from Poisson processes,
 that $N_{\lambda p_{i0}}$ and $M_{\lambda p_{i1}}$ are independent and Poisson distributed with
means $\lambda p_{i0}$ and $\lambda p_{i1}$.

Note that $V_Z^i(n)=\bo( n \log n)$ and the Lipschitz continuity of $\nu_Z^i$ imply that, as $\lambda\rightarrow\infty$
\begin{align}\label{pois_var_upper_bound}
\V( Z_{N_\lambda }^i )=\E[V_Z^i(N_\lambda)]+\V(\nu_Z^i(N_\lambda)) =\bo(\lambda \log \lambda),\quad i\in\Sigma,
\end{align}
where $\E[N_\lambda \log(N_\lambda)]=\bo(\lambda \log \lambda)$ is not hard to check (details are given in the appendix, lemma \ref{lemPoisAs}).
Moreover, Lemma \ref{lem_asym_etani} implies, as $\lambda\rightarrow\infty$
\begin{align}\label{pois_toll_upper_bound}
\V(\eta_{N_\lambda}^{i,2})=\bo\left(\E[ (\log (N_\lambda+1))^2\right)=\bo\left(\sqrt{\lambda}\right),
\end{align}
where the second bound holds since $(\log(n+1))^2=\bo(\sqrt n)$ and $\E[\sqrt{N_\lambda}]=\bo(\sqrt{\lambda})$ as $\lambda\rightarrow\infty$ (details are given in the appendix,
lemma \ref{lemPoisAs}).
Hence, (\ref{pois_rec_small}) implies for $\widetilde{V}_i(\lambda):=\V(Z_{N_\lambda}^i)$
\begin{align}
\widetilde{V}_i(\lambda)&=\V(Z_{N_{\lambda p_{i0}}}^0+ Z_{M_{\lambda p_{i1}}}^{1} + \eta_{N_\lambda}^{i,2})\nonumber\\
&=\V(Z_{N_{\lambda p_{i0}}}^0+ Z_{M_{\lambda p_{i1}}}^{1})+\bo\left( \lambda^{3/4}\sqrt{\log \lambda}\right)\nonumber\\
&= \widetilde{V}_0(\lambda p_{i0})+ \widetilde{V}_{1}(\lambda p_{i1})+\bo\left( \lambda^{3/4}\sqrt{\log \lambda)}\right).\label{pois_var_rec}
\end{align}
in which the second equality holds by (\ref{pois_var_upper_bound}), (\ref{pois_toll_upper_bound}) and Lemma \ref{lem_var_sum}
and the last equality holds since $Z_{N_{\lambda p_{i0}}}^0$ and $Z_{M_{\lambda p_{i1}}}^{1}$ are independent (which is one of the reason for poissonization).

Lemma \ref{lem_asym_pois_var} yields the refined upper bound
\begin{align}\label{pois_var_bound}
\widetilde{V}_i(\lambda)=\bo(\lambda).
\end{align}
Finally, we need to deduce asymptotic results for $V_Z^i(n)$ out of (\ref{pois_var_bound}).
Since we have for both $i\in\Sigma$
$$\V(Z_{N_\lambda }^i)= \E[V_Z^i(N_\lambda ) ] + \V(\nu_Z^i(N_\lambda))$$
and, by the Lipschitz continuity of $\nu_Z^i$ that $\V(\nu_Z^i(N_\lambda))=\bo(\lambda)$,
we may conclude that, as $\lambda\rightarrow\infty$
\begin{align}\label{pois_var_pf_1}\E[V_Z^i(N_\lambda ) ]=\bo(\lambda).\end{align}
In order to apply Lemma \ref{depois_lem} we need to check that
\begin{align}\label{zz_fuer_depois}
\Delta V_Z^i(n)=\bo(\sqrt{n})
\end{align}
which may be done by the transfer theorem \ref{lem_transf_inc}:
First note that (\ref{var_rec_small}) and
Lemma \ref{lemdelta_app} imply for the differences
$$\Delta V_Z^i(n)= p_{i0} \E[\Delta V_Z^0(I_n^i)] + (1-p_{i0})\E[\Delta V_Z^{1}(n-I_n^i)]+\varepsilon_i(n),$$
where $\varepsilon_i$ is given by
$$\varepsilon_i(n)=\V(\nu_Z^0(I_{n+1}^i)+\nu_Z^{1}(n+1-I_{n+1}^i))-\V(\nu_Z^0(I_n^i)
+\nu_Z^{1}(n-I_n^i)).$$
The Lipschitz-continuity of $\nu_Z^i$ yields $\V(\nu_Z^0(I_n^i)+\nu_Z^{1}(n-I_n^i))=\bo(n)$. Moreover,
\begin{align*}
 \V(\nu_Z^0(I_{n+1}^i)+\nu_Z^{1}(n+1-I_{n+1}^i))= \V\left( \nu_Z^0(I_n^i)+\nu_Z^1(n-I_n^i) + B \Delta \nu_Z^0(I_n^i) + (1-B) \Delta \nu_Z^1(n-I_n^i)\right)
\end{align*}
where $B$ is independent of $I_n^i$ and Bernoulli distributed with parameter $p_{i0}$. Since $\Delta \nu_Z^0$ and $\Delta \nu_Z^1$ are bounded,
we may conclude by lemma \ref{lem_var_sum} that
\begin{align*}
 \V(\nu_Z^0(I_{n+1}^i)+\nu_Z^{1}(n+1-I_{n+1}^i))=\V(\nu_Z^0(I_n^i)+\nu_Z^{1}(n-I_n^i))+\bo(\sqrt n)
\end{align*}
which implies $\varepsilon_i(n)=\bo(\sqrt n)$ and therefore, $\Delta V_Z^i(n)=\bo(\sqrt n)$ by lemma \ref{lem_transf_inc}.
Hence, the depoissonization lemma \ref{depois_lem} is applicable and the assertion follows.
\end{proof}
We finish the section with the proof of theorem \ref{thm:variance}:
\begin{proof}[Proof of theorem \ref{thm:variance}]
Recall that for $n\in\N_0$, $i\in\Sigma$ we have
\begin{align*}
\rho_i:=p_{i0}\delta_0+p_{i1}\delta_1,\qquad B_n^i:=B_n^{\rho_i}.
\end{align*}
Moreover, we define for $n\in\N_0$ and $i\in\Sigma$
$$V_i(n):=\V(B_n^i), \qquad \nu_i(n):=\E[B_n^{i}].$$
 We start with the proof for the special cases $\mu=\rho_i$, $i\in\Sigma$. In these cases we have by definition of $(X_n^i, Z_n^i)_{n\geq 0, i\in\Sigma}$ that
\begin{align}\label{V_i_pf_var}
V_i(n)=\V(X_n^i+Z_n^i)= \sigma^2 n \log n + \bo \left(n \sqrt{\log n}\right).
\end{align}
where the last equality holds by Lemma \ref{lem_var_changed_toll}, \ref{lem_var_small} and \ref{lem_var_sum}.

In order to obtain the result for arbitrary initial distributions $\mu$ recall that, by (\ref{decomp_initial}),
$$B_{n}^\mu=B_{K_n^\mu}^0+B_{n-K_n^\mu}^{1}+n$$
where $K_n^\mu$ is binomial $B(n,\mu(0))$ distributed.

Hence, we have by the independence of $(B_n^0)_{n\geq 0}$, $(B_n^1)_{n\geq 0}$ and $(K_n^\mu)_{n\geq0}$
\begin{align*}
 \V(B_{n}^\mu)&= \E[V_0(K_n^\mu)]+\E[V_{1}(n-K_n^\mu)] + \V(\nu_0(K_n^\mu)+\nu_{1}(n-K_n^\mu))\\
&=\sigma^2 \E [K_n^\mu\log K_n^\mu +(n-K_n^\mu)\log(n-K_n^\mu) ] + \V(\nu_0(K_n^\mu)+\nu_{1}(n-K_n^\mu)) \\
&\quad+ \bo \left(n \sqrt{\log n}\right)
\end{align*}
where the second equality holds by (\ref{V_i_pf_var}).
Therefore, it only remains to show that
\begin{align}
 \E [K_n^\mu\log K_n^\mu +(n-K_n^\mu)\log(n-K_n^\mu) ]&= n \log n + \bo \left(n \sqrt{\log n}\right),\label{thm_var_zz_1}\\
\V(\nu_0(K_n^\mu)+\nu_{1}(n-K_n^\mu))&= \bo \left(n \sqrt{\log n}\right).\label{thm_var_zz_2}
\end{align}
For (\ref{thm_var_zz_1}) note that $x\mapsto x \log x +(1-x)\log (1-x)$ is bounded on $[0,1]$ (with $0 \log 0:=0$). Therefore,
we have
\begin{align*}
 &\E [K_n^\mu\log K_n^\mu +(n-K_n^\mu)\log(n-K_n^\mu) ]-n\log n \\
=&n\E [K_n^\mu/n\log (K_n^\mu /n) +(1-K_n^\mu/n)\log(1-K_n^\mu/n) ]\\
=&\bo(n)
\end{align*}
which implies (\ref{thm_var_zz_1}).
Note that by Proposition \ref{thEwLip} we have for $i\in\Sigma$ and $n\in\N_0$
$$\nu_i(n)= \frac 1 H n \log n + f_i(n)$$
where $f_0$ and $f_1$ are Lipschitz continuous functions. Since the Lipschitz continuity implies
$\V(f_0(K_n^\mu)+f_{1}(n-K_n^\mu))=\bo(n)$, it only remains to show that
$$\V( K_n^\mu \log K_n^\mu + (n-K_n^\mu) \log (n-K_n^\mu))=\bo (n),$$
which is an easy computation and essentially covered by the proof of lemma \ref{lem_var_changed_toll}. Thus, we leave the details to the reader.
\end{proof}

\subsection{Existence of the Splitting}\label{sec:exsplitting}

In the analysis of the variance we work with pairs $(X_n^i, Z_n^i)_{n\in\N_0}$, $i\in\Sigma$, that satisfy the initial
conditions
\begin{align}\label{split_initial_app}
 X_n^i=Z_n^i=0, \quad i\in\Sigma, n\leq 1,
\end{align}
as well as the stochastic recurrences
\begin{equation}
\label{split_for_var_app}
\begin{pmatrix}  X_{n}^{i} \\ Z_{n}^{i} \end{pmatrix} \stackrel{d}{=}
\begin{pmatrix}  X_{I_n^i}^{0} \\ Z_{I_n^i}^{0} \end{pmatrix}+\begin{pmatrix}  X_{n-I_n^i}^{1} \\ Z_{n-I_n^i}^{1} \end{pmatrix}
+\begin{pmatrix}  \eta_n^{i,1} \\ \eta_n^{i,2} \end{pmatrix},\quad n\geq 2,i\in\Sigma
\end{equation}
where $(X_0^0,\ldots,X_n^0,Z_0^0,\ldots,Z_n^0)$, $(X_0^1,\ldots,X_n^1,Z_0^1,\ldots,Z_n^1)$ and $I_n^i$
are independent, $\eta_n^{i,2}=n-\eta_n^{i,1}$ and $\eta_n^{i,1}$ is some constant satisfying
$$\eta_n^{i,1}=0,\quad n\leq 1 \qquad \text{and} \qquad \eta_n^{i,1}=\bo(n) \quad (n\rightarrow\infty).$$
We now discuss how to get $(X_n^i,Z_n^i)_{n\in\N_0, i\in\Sigma}$ with finite second moment that satisfy
(\ref{split_initial_app}) and (\ref{split_for_var_app}) as well as
\begin{align}\label{moments_split}
\E[X_n^i+Z_n^i]=\E[B_n^i]\qquad \text{and}\qquad\V(X_n^i+Z_n^i)=\V(B_n^i), \quad n\in\N_0,\; i\in\Sigma.
\end{align}

By iterating (\ref{split_for_var_app}) on the right hand side
one expects
\begin{align}\label{def_split}\begin{pmatrix} X_n^i \\ Z_n^i \end{pmatrix}
\stackrel{d}{=} \begin{pmatrix} \eta_n^{i,1}+\sum_{k=1}^\infty \sum_{I:=(i_1,\ldots,i_k)\in\{0,1\}^k} \eta_{J_{i}^{I} (n)}^{i_k, 1} \\
\eta_n^{i,2}+\sum_{k=1}^\infty  \sum_{I:=(i_1,\ldots,i_k)\in\{0,1\}^k} \eta_{J_{i}^{I} (n)}^{i_k, 2}
 \end{pmatrix}
\end{align}
where $J_i^I(n)$ is some iteration of binomial distributed random variables that is generated as follows: For $n\in\N_0$ and $i\in\Sigma$
let $I_i(n):=\sum_{j=1}^n L_j^i$ where $(L_j^i)_{j\in\N}$ is a sequence of independent Bernoulli $B(p_{i0})$ distributed
random variables. Moreover, for each $k\geq 1$, $i\in\Sigma$ and $I\in\{0,1\}^k$ let $(I_{i,0}^I(n), I_{i,1}^I(n))_{n\geq 0}$
be an independent copy of $(I_i(n), n-I_i(n))_{n\geq 0}$.
Then we define for both $i\in\Sigma$
$$J_i^{(0)}(n):=I_{i}(n),\quad J_i^{(1)}(n)=n-I_{i}(n),$$
and, for $k\geq 2$ and $I=(i_1,\ldots,i_k)\in\{0,1\}^k$
$$J_i^I(n):=I_{i_{k-1},i_k}^{(i_1,\ldots.i_{k-1})}\left(J_i^{(i_1,\ldots,i_{k-1})}(n)\right).$$
In the context of radix sort $J_i^I(n)$ may be interpreted as the number of strings with prefix $I$ among $n$ i.i.d.~strings generated by a Markov source.\\

\noindent
Now let $\tau_i(n) := \min\{k\geq 1 \,:\, J_i^I(n) \leq 1 \text{ for all } I\in\{0,1\}^k\}$. Since $\eta_n^{i,1}=\eta_n^{i,2}=0$ for
$n\leq 1 $ and $i\in \{0,1\}$, note that all summands for $k\geq \tau_i$ equal zero in (\ref{def_split}). Hence, if we have
$\tau_i(n)<\infty$ then the sum in (\ref{def_split}) is finite.

We will now discuss that for every $n\in\N$ we have $\tau_i(n)<\infty$ almost surely and then use (\ref{def_split}) to define $(X_n^i,Z_n^i)$
and finally check that (\ref{split_for_var_app}) and (\ref{moments_split}) holds. To this end note that
$$M_k^i(n):=\max\{J_i^I(n) \,:\, I\in\{0,1\}^k\}$$
is bounded by $n$, non-increasing in $k$ and for $M_k(n)\geq 2$ the probability that $M_k(n)$ decreases by at least one (i.e.~$M_{k+1}(n)\leq M_k(n)-1$) is
at least $(2p(1-p))^{n/2}$, $p:=\max\{p_{ij}|i,j\in\Sigma\}$, which can be seen as follows:
At each step $k$ there are at most $n/2$ indices $I_1,\ldots, I_{n/2}\in\{0,1\}^k$ with $J_i^{I_j}(n)\geq 2$ since
we have
$$\sum_{I\in\{0,1\}^k} J_i^I(n)=n.$$
For each of these indices $I_j=(i_1^j,\ldots, i_k^j)$ the probability that the next binomial splitter decreases
$\max\{J_i^{(i_1^j,\ldots,i_k^j,0)},J_i^{(i_1^j,\ldots,i_k^j,1)}\}$ by at least one is at least $2p(1-p)$ since
starting the underlying Bernoulli chain of $(I_{i_{k}^j,0}^{I_j}(m))_{m\geq 0}$ with $01$ or $10$
causes a decrease. By the independence of $(I_{i_k^1,0}^{I_1}(m))_{m\geq 0},\ldots (I_{i_k^{n/2},0}^{I_{n/2}}(m))_{m\geq 0}$
we obtain the upper bound $(2p(1-p))^{n/2}$.

This yields that $\tau_i(n)$ is stochastically dominated by a negative
binomial \mbox{$nB(n,(2p(1-p))^{n/2})$} distributed random variable.
In particular, we have for all $n\in\N$
$$\E[\tau_i(n)]\leq \frac {n} {(2p(1-p))^{n/2}} <\infty \qquad \text{and} \qquad \V(\tau_i)<\infty.$$
This implies that mean and variance of $X_n^i$ and $Z_n^i$ defined by (\ref{def_split}) are finite since
we have $|\eta_n^{i,1}|\leq Cn$ for some constant $C>0$ which together with $\sum_{I\in\{0,1\}^k} J_i^I(n)=n$ yields
$$\E[|X_n^i|]\leq |\eta_n^{i,1}| + C n \E[|\tau_i(n)|]<\infty,\qquad \V(X_n^i)\leq \E[(|\eta_n^{i,1}|+Cn \tau_i(n))^2]<\infty$$
and similar bounds for $Z_n^i$ since $\eta_n^{i,2}=\bo(n)$.

Hence, it only remains to show that the definition (\ref{def_split}) implies (\ref{split_for_var_app}) and (\ref{moments_split}).
But (\ref{split_for_var_app}) holds by construction and is not hard to check. For (\ref{moments_split}) note that (\ref{split_initial_app})
and (\ref{split_for_var_app}) implies for the sum $S_n^i:=X_n^i+Z_n^i$ in the case $d=0$ that for both $i\in\Sigma$
$$S_n^i=0, \quad n\leq 1 \qquad\text{ and } \qquad S_n^i\stackrel{d}{=} S_{I_n^i}^0+S_{n-I_n^i}^{1}+n$$
which uniquely defines all moments of $S_n^i$ that are finite. Since $B_n^i$ satisfies the same conditions we obtain
$$\E[S_n^i]=\E[B_n^i]\qquad \text{and} \qquad \V(S_n^i)=\V(B_n^i).$$

\section{Asymptotic Normality} \label{sec:limit}

Our main result is the asymptotic normality of the number of bucket operations:
\begin{thm} \label{thm:limit}
For the number $B_n^\mu$ of bucket operations under the Markov source model with conditions (\ref{cond_prob}) we have
\begin{align}\label{limit_law_app}
 \frac{B_n^\mu - \E[B_n^\mu]}{\sqrt{n\log n}} \stackrel{d}{\longrightarrow} {\cal N}(0,\sigma^2), \qquad (n\to\infty),
\end{align}
where $\sigma^2>0$ is independent of the initial distribution $\mu$ and given by (\ref{var_formula}).
\end{thm}

As in the analysis of the mean, we first derive limit laws for $B_n^0$ and $B_n^1$ and then transfer these to a limit law for $B_n^\mu$ via (\ref{decomp_initial}). We abbreviate for $i\in\Sigma$ and $n\in\N_0$
\begin{align*}
\nu_i(n):=\E[B_n^i],   \qquad \sigma_i(n):=\sqrt{\V(B_n^i)}.
\end{align*}
Note that we have $\nu_i(0)=\nu_i(1)=\sigma_i(0)=\sigma_i(1)=0$ and $\sigma_i(n)>0$ for all $n\ge 2$.
We define the standardized variables by
\begin{align}\label{normal_rec}
Y^i_n := \frac{B_n^i - \E[B_n^i]}{\sigma_i(n)}, \qquad i\in\Sigma, n\ge 2,
\end{align}
and $Y^i_0:=Y^i_1:=0$.

Our proof if based on an application of the contraction method to the recursive distributional system (\ref{decomp_zero})--(\ref{decomp_one}). The Zolotarev metric used here has been
studied in the context of the contraction method systematically in \cite{NeRu04}. We only need
the following properties, see Zolotarev \cite{zo76,zo77}:
For distributions
${\cal L}(X)$, ${\cal L}(Y)$ on $\R$ the Zolotarev distance $\zeta_s$, $s>0$, is defined by

\begin{equation}
\label{eq:3.6_app} \zeta_s(X,Y) := \zeta_s({\cal L}(X),{\cal L}(Y)):=\sup_{f\in {\cal F}_s}|\E[f(X) -
f(Y)]|
\end{equation}
where $s=m+\alpha$ with $0<\alpha\le 1$,
$m\in\N_0$, and
\begin{equation}
{\cal F}_s:=\{f\in
C^m(\R,\R):\|f^{(m)}(x)-f^{(m)}(y)\|\le
\|x-y\|^\alpha\},
\end{equation}
 the space of $m$ times
continuously differentiable functions from
$\R$ to $\R$ such that the $m$-th
derivative is H\"older continuous of order
$\alpha$ with H\"older-constant $1$.
We have that $\zeta_s(X,Y)<\infty$, if all
moments of orders $1,\ldots,m$ of $X$ and $Y$ are
equal and if the $s$-th absolute moments of $X$ and
$Y$ are finite.  Since later on only the case $2<s\le 3$ is
used, for finiteness of  $\zeta_s(X,Y)$ it is thus sufficient  for these
$s$ that  mean and variance of
$X$ and $Y$ coincide and both have a finite absolute moment of
 order $s$.

\noindent
{\bf Properties of $\zeta_s$}:
(1)  Convergence in $\zeta_s$ implies weak convergence on $\R$.
\\
(2) $\zeta_s$ is $(s,+)$ ideal, i.e., we have
\begin{eqnarray*}
\zeta_s(X+Z,Y+Z)\le\zeta_s(X,Y), \quad  \zeta_s(cX,cY) = c^s \zeta_s(X,Y)
\end{eqnarray*}
for all  $Z$ being independent of $(X,Y)$ and all $c>0$.

We will use an upper bound of $\zeta_s$ by the minimal $L_p$ metric $\ell_p$.
For distributions
${\cal L}(X)$, ${\cal L}(Y)$ on $\R$ and  $p>0$ we have
\begin{eqnarray*}
\ell_p(X,Y):= \ell_p({\cal L}(X),{\cal L}(Y)):= \inf\left\{\|X'-Y'\|_p\;:\;X'\stackrel{d}{=} X,    Y'\stackrel{d}{=} Y\right\},
\end{eqnarray*}
where $\|X\|_p:=(\E \|X\|^p)^{(1/p) \wedge 1}$ denotes the $L_p$ norm. We have  $\ell_p(X,Y)<\infty$ if $\|X\|_p, \|Y\|_p<\infty$. The bound  used later  for $2<s\le 3$ is, see Lemma 5.7 in \cite{DrJaNe08},
\begin{equation}\label{ralph1_app}
\zeta_s(X,Y)\le \left( (\E \|X\|^s)^{1-1/s} +(\E \|Y\|^s)^{1-1/s}
  \right)\ell_s(X,Y),
\end{equation}
for all $X$ and $Y$ with joint mean and variance and finite absolute moments of order $s$.

\begin{prop} \label{conv_zolo_prop_app}
For both sequences $(Y^i_n)_{n\ge 0}$, $i\in\Sigma$, we have for all $2<s\le 3$
\begin{align} \label{conv_zolo_app}
\zeta_s\left(Y^i_n, {\cal N}(0,1)\right) \to 0,\qquad (n\to\infty).
\end{align}
\end{prop}
\begin{proof}
>From the recurrences (\ref{decomp_zero_one}) and the normalization (\ref{normal_rec})
we obtain for $i\in\Sigma$
 \begin{align}\label{mod_rec}
Y^i_n \stackrel{d}{=}\frac{\sigma_0(I_n^i)}{\sigma_i(n)} Y^0_{I_n^i}+ \frac{\sigma_{1}(n-I_n^i)}{\sigma_i(n)}Y^{1}_{n-I_n^i} + b_i(n), \qquad n\ge 2,
\end{align}
where
\begin{align*}
 b_i(n)= \frac{1}{\sigma_i(n)}\left( n +\nu_0(I_n^i) + \nu_{1}(n-I^i_n)-\nu_i(n) \right),
\end{align*}
and in (\ref{mod_rec}) we have that $(Y^0_0,\ldots,Y^0_n)$, $(Y^1_0,\ldots,Y^1_n)$ and
$(I_n^0,I_n^1)$ are independent.

For independent normal ${\cal N}(0,1)$ distributed random variables
$\No_0,\No_1$ also independent of  $(I_n^0,I_n^1)$ we define
\begin{align} \label{acc_seq_app}
Q^i_n:=\frac{\sigma_0(I_n^i)}{\sigma_i(n)} \No_0 + \frac{\sigma_{1}(n-I_n^i)}{\sigma_i(n)}\No_{1} + b_i(n), \qquad n\ge 2.
\end{align}
Note that we have $\E[Q^i_n]=0$ and $\V(Q^i_n)=1$ for all $n\ge 2$. For the variance, this
is seen by conditioning on $I_n^i$ in (\ref{mod_rec}) and (\ref{acc_seq_app}) and using that
$Y^i_j$ and $\No_i$ have the same variance $1$ for all $j\ge 2$ and that for $j\in\{0,1\}$
the coefficients $\sigma_0(j)/\sigma_i(n)$ are zero, whereas for  $j\in\{n-1,n\}$ the coefficients $\sigma_{1}(n-j)/\sigma_i(n)$ are zero.  Hence, the Zolotarev
distances $\zeta_s(Y^i_n,Q^i_n)$, $\zeta_s(Q^i_n,\No_i)$
and $\zeta_s(Y^i_n,\No_i)$ are
finite for all $n\ge 2$ and $i\in\Sigma$, where we have $2<s\le 3$.

We denote by $\No$ another normal ${\cal N}(0,1)$ distributed random variable. Then we have
\begin{align*}
\zeta_s(Y_n^i,\No) \le \zeta_s(Y_n^i,Q_n^i) + \zeta_s(Q_n^i,\No).
\end{align*}
In the first step we show that $\zeta_s(Q_n^i,\No)\to 0$ as $n\to\infty$ for both $i\in\Sigma$.
Note that $\|Q_n^i\|_s$ is uniformly bounded in $n\ge 2$ and $i\in\Sigma$. Hence, by (\ref{ralph1_app}) there exists a constant $C>0$ such that
 $\zeta_s(Q_n^i,\No) \le C \ell_s(Q_n^i,\No)$. Thus, it is sufficient to show
$\ell_s(Q_n^i,\No)\to 0$. With
\begin{align*}
\No \stackrel{d}{=} \sqrt{p_{i0}} \No_0 + \sqrt{1-p_{i0}} \No_{1}
\end{align*}
we obtain
\begin{align}\label{ineq_1_app}
\ell_s(Q_n^i,\No)\le \left\| \left(\frac{\sigma_0(I_n^i)}{\sigma_i(n)} -\sqrt{p_{i0}}\right)\No_0 \right\|_s
+\left\| \left(\frac{\sigma_{1}(n-I_n^i)}{\sigma_i(n)}-\sqrt{1-p_{i0}}\right)\No_{1} \right\|_s + \|b_i(n)\|_s.
\end{align}
For the first summand in (\ref{ineq_1_app}) we have, by the strong law of large numbers and the variance expansion (\ref{var_asympt})  that $\sigma_0(I_n^i)/\sigma_i(n)\to \sqrt{p_{i0}}$ almost surely. Since $\No_0$ is independent from $I_n^i$ and $\|\No_0\|_s<\infty$ we obtain from dominated convergence that this first summand tends to zero. By similar arguments we also have that the second summand in (\ref{ineq_1_app}) tends to zero.
 The third summand $\|b_i(n)\|_s$ is bounded as
follows: With the notation (\ref{def_fi}) and $h(x)=x\log x$ as in
Lemma \ref{lemBinAs_app} of the Appendix, we have
\begin{align*}
b_i(n)=\frac{1}{\sigma_i(n)}\Big( & \frac{n}{H}\left\{h(I_n^i/n)-\E[h(I_n^i/n)] +
h((n-I_n^i)/n)-\E[h((n-I_n^i)/n)] \right\} \\
&~+f_0(I_n^i)-\E[f_0(I_n^i)] +
f_{1}(n-I_n^i) - \E[f_{1}(n-I_n^i)] \Big)
\end{align*}
With $\sigma_i(n)=\Omega(\sqrt{n\log n})$ and (\ref{ccc_app}) the contributions of all
summands involving the function $h$ are $\bo(1/\sqrt{\log n})$ in the  $L_s$-norm, hence we have
\begin{align*}
\|b_i(n)\|_s \le & \;\|f_0(I_n^i)-\E[f_0(I_n^i)]\|_s + \| f_{1}(n-I_n^i) - \E[f_{1}(n-I_n^i)]\|_s\\
 &~+\bo(1/\sqrt{\log n}), \quad (n\to\infty).
\end{align*}
Furthermore, to bound $\|f_0(I_n^i)-\E[f_0(I_n^i)]\|_s$ we use an independent copy $H_n^i$ of $I_n^i$. Then, by Jensen's inequality for conditional expectations
and the Lipschitz property of $f_i$ in Proposition \ref{thEwLip} (with Lipschitz constant bounded by $C$)
\begin{align}
\|f_0(I_n^i)-\E[f_0(I_n^i)]\|_s &= \|\E[f_0(I_n^i)-f_0(H_n^i)\,|\,I_n^i] \|_s  \nonumber \\
&\le \|f_0(I_n^i)-f_0(H_n^i)\|_s   \nonumber \\
&\le C  \|I_n^i-H_n^i\|_s   \nonumber \\
&\le 2C \|I_n^i-\E[I_n^i]\|_s  \nonumber \\
&= \bo(\sqrt{n}). \label{est_cond_app}
\end{align}
Since $\| f_{1}(n-I_n^i) - \E[f_{1}(n-I_n^i)]\|_s$ is bounded analogously and $\sigma_i(n)=\Omega(\sqrt{n\log n})$ we obtain altogether  as  $n\to\infty$
and for both $i\in\Sigma$.
\begin{align*}
\|b_i(n)\|_s = \bo\left(\frac{1}{\sqrt{\log n}}\right).
\end{align*}
This completes the estimate for the first step $\zeta_s(Q_n^i,\No)\to 0$ as $n\to\infty$.

Now, we denote the distances $d_i(n):=\zeta_s(Y_n^i,\No)$, for $n\ge 2$, and $d_i(0):=d_i(1):=0$  for $i\in\Sigma$. Conditioning on $I_n^i$ and using that $\zeta_s$ is $(s,+)$ ideal we obtain for all $n\ge 2$
\begin{align}
\lefteqn{d_i(n)}  \nonumber \\
&\le \zeta_s(Y_n^i,Q_n^i) + o(1) \nonumber\\
&=\zeta_s\left(\frac{\sigma_0(I_n^i)}{\sigma_i(n)} Y^0_{I_n^i}+ \frac{\sigma_{1}(n-I_n^i)}{\sigma_i(n)}Y^{1}_{n-I_n^i} + b_i(n),
\frac{\sigma_0(I_n^i)}{\sigma_i(n)} \No_0 +
\frac{\sigma_{1}(n-I_n^i)}{\sigma_i(n)}\No_{1} + b_i(n)  \right)+ o(1)\nonumber \\
&\le \sum_{j=0}^n {n \choose j} p_{i0}^j (1-p_{i0})^{n-j} \zeta_s\left(\frac{\sigma_0(j)}{\sigma_i(n)} Y^0_{j}+ \frac{\sigma_{1}(n-j)}{\sigma_i(n)}Y^{1}_{n-j} + {\bf 1}_{\{I_n^i=j\}}b_i(n),   \right. \nonumber\\
& \left. \hspace{4.5cm}
\frac{\sigma_0(j)}{\sigma_i(n)} \No_0 + \frac{\sigma_{1}(n-j)}{\sigma_i(n)}
\No_{1} + {\bf 1}_{\{I_n^i=j\}}b_i(n)  \right)+ o(1)  \nonumber\\
&\le \sum_{j=2}^n {n \choose j} p_{i0}^j (1-p_{i0})^{n-j}\left\{ \left(\frac{\sigma_0(j)}{\sigma_i(n)}\right)^s \zeta_s(Y_j^0,\No_0) + \left(\frac{\sigma_{1}(n-j)}{\sigma_i(n)}\right)^s \zeta_s(Y_{n-j}^{1},\No_{1})\right\} + o(1)  \nonumber\\
&= \E \left[ \left(\frac{\sigma_0(I_n^i)}{\sigma_i(n)}\right)^s d_0(I_n^i) +  \left(\frac{\sigma_{1}(n-I_n^i)}{\sigma_i(n)}\right)^s d_{1}(n-I_n^i)  \right]+ o(1). \label{basic_est_app}
\end{align}
With $d(n):=d_0(n)\vee d_1(n)$ we obtain for both $i\in\Sigma$ that
\begin{align} \label{dist_est_app}
d_i(n)
\le &\; \E \left[{\bf 1}_{\{ 1\le I_n^i \le n-1  \}}\left\{ \left(\frac{\sigma_0(I_n^i)}{\sigma_i(n)}\right)^s +  \left(\frac{\sigma_{1}(n-I_n^i)}{\sigma_i(n)}\right)^s  \right\} \right] \sup_{1\le j\le n-1} d(j)  \\ &~
+ ((1-p_{i0})^n+ p_{i0}^n) d(n) +o(1). \nonumber
\end{align}
With
\begin{align*}
\xi(n)&:= \max_{i\in\Sigma} \E \left[{\bf 1}_{\{ 1\le I_n^i \le n-1  \}}\left\{ \left(\frac{\sigma_0(I_n^i)}{\sigma_i(n)}\right)^s +  \left(\frac{\sigma_{1}(n-I_n^i)}{\sigma_i(n)}\right)^s  \right\} \right],  \\
\varepsilon(n)&:=  \max_{i\in\Sigma}\left\{(1-p_{i0})^n+ p_{i0}^n\right\}
\end{align*}
we obtain by taking the maximum of the right hand sides in (\ref{dist_est_app})
\begin{align} \label{dist_est_2_app}
d(n)
\le  \frac{\xi(n)}{1-\varepsilon(n)} \sup_{1\le j\le n-1} d(j) +o(1).
\end{align}
We have $\varepsilon(n)\to 0$ and, since $s>2$ and $p_{ii}\in(0,1)$ for both $i\in\Sigma$,
\begin{align}\label{limit_cont_app}
\xi &:= \lim_{n\to\infty} \xi(n) = \max_{i\in\Sigma} \left\{ p_{i0}^{s/2} + (1-p_{i0})^{s/2}\right\} <1.
\end{align}
With (\ref{dist_est_2_app}) this implies that $(d(n))_{n\ge 1}$ remains bounded. We denote
$\varrho:=\sup_{n \ge 0} d(n) $ and $\eta:=\limsup_{n\to\infty} d(n)$. Hence, we have $\varrho, \eta < \infty$ and for any $\varepsilon>0$ there exists an $n_0\ge 2$ such that for all $n\ge n_0$ we have $d(n)\le \eta + \varepsilon$. From (\ref{basic_est_app})
we obtain with (\ref{limit_cont_app}) for both $i\in\Sigma$
\begin{align}
d_i(n)
\le &\;\E \left[{\bf 1}_{\{I_n^i <n_0\} \cup \{I_n^i > n-n_0\}} \left\{
\left(\frac{\sigma_0(I_n^i)}{\sigma_i(n)}\right)^s  +
\left(\frac{\sigma_{1}(n-I_n^i)}{\sigma_i(n)}\right)^s   \right\}\right]\varrho\\
&~+\E \left[{\bf 1}_{\{n_0\le I_n^i \le n-n_0\}} \left\{\left(\frac{\sigma_0(I_n^i)}{\sigma_i(n)}\right)^s  +  \left(\frac{\sigma_{1}(n-I_n^i)}{\sigma_i(n)}\right)^s   \right\}\right](\eta+\varepsilon)  +o(1)\\
\le &\; (\xi+o(1))(\eta+\varepsilon) + o(1)
\end{align}
with appropriate $o(1)$ terms. Maximizing over $i\in\Sigma$ this yields
$d(n)\le o(1)+(\xi+o(1))(\eta+\varepsilon)$ and with $n\to\infty$
\begin{align*}
\eta \le \xi (\eta+\varepsilon).
\end{align*}
Since $\varepsilon>0$ can be chosen arbitrarily small we obtain $\eta=0$, i.e.~$\zeta_s(Y_n^i,\No)\to 0$ as $n\to\infty$ for both $i\in \Sigma$.
\end{proof}

\begin{proof}[Proof of Theorem \ref{thm:limit}]
We write
\begin{align*}
\frac{B_n^\mu - \E[B_n^\mu]}{\sqrt{n\log n}}
\stackrel{d}{=}\frac{B_{K_n}^0 - \nu_0(K_n)+ B_{n-K_n}^1- \nu_1(n-K_n)}{\sqrt{n\log n}} + \frac{\nu_0(K_n)+ \nu_1(n-K_n)+n - \E[B_n^\mu]}{\sqrt{n\log n}}.
\end{align*}
By the Lemma of Slutzky it is sufficient to show, as $n\to\infty$,
\begin{align}
\frac{B_{K_n}^0 - \nu_0(K_n)+ B_{n-K_n}^1- \nu_1(n-K_n)}{\sqrt{n\log n}} & \stackrel{d}{\longrightarrow} {\cal N}(0,\sigma^2)  \label{conv_1_app}\\
 \frac{\nu_0(K_n)+ \nu_1(n-K_n)+n - \E[B_n^\mu]}{\sqrt{n\log n}} & \stackrel{\Prob}{\longrightarrow} 0.\label{conv_2_app}
\end{align}
For showing (\ref{conv_1_app}) note that by Proposition \ref{conv_zolo_prop_app} $(B_n^i - \E[B_n^i])/\sqrt{n\log n} \to {\cal N}(0,\sigma^2) $ in distribution for both $i\in\Sigma$.
We set $A_n:= [\mu_0 n -n^{2/3}, \mu_0 n +n^{2/3}] \cap \N_0$ and $A_n^c:=\{0,\ldots,n\}\setminus A_n$. Then by Chernoff's bound (or the central limit theorem) we have $\Prob(K_n \in A_n) \to 1$. For all $x\in\R$ we have
\begin{align*}
\lefteqn{\Prob\left(\frac{B_{K_n}^0 - \nu_0(K_n)+ B_{n-K_n}^1- \nu_1(n-K_n)}{\sqrt{n\log n}} \le x\right)}\\
&= o(1) + \sum_{j\in A_n} \Prob(K_n=j) \Prob\left(\frac{B_{j}^0 - \nu_0(j)}{\sqrt{n\log n}} +  \frac{B_{n-j}^1   - \nu_1(n-j)}{\sqrt{n\log n}} \le x\right).
\end{align*}
For $j\in A_n$ we have $\sqrt{j\log j}/\sqrt{n\log n} \to \sqrt{\mu_0}$ and
$\sqrt{(n-j)\log(n- j)}/\sqrt{n\log n} \to \sqrt{1-\mu_0}$. Hence, we have $(B_{j}^0 - \nu_0(j))/\sqrt{n\log n} \to
{\cal N}(0,\mu_0 \sigma^2)$ and $(B_{n-j}^1   - \nu_1(n-j))/\sqrt{n\log n}\to
{\cal N}(0,(1-\mu_0) \sigma^2)$ in distribution and the two summands are independent. Together, denoting by $N_{0,\sigma^2}$ an  ${\cal N}(0,\sigma^2)$ distributed random variable we obtain
\begin{align*}
\Prob\left(\frac{B_{K_n}^0 - \nu_0(K_n)+ B_{n-K_n}^1- \nu_1(n-K_n)}{\sqrt{n\log n}} \le x\right)
&= o(1) + \sum_{j\in A_n} \Prob(K_n=j) ( \Prob\left(N_{0,\sigma^2} \le x\right) +o(1))\\
&\to \Prob\left(N_{0,\sigma^2} \le x\right),
\end{align*}
where the latter convergence is justified by dominated convergence. This shows (\ref{conv_1_app}).

For (\ref{conv_2_app}) note that (\ref{decomp_initial}) implies
$$\E[B_n^\mu] = \E[\nu_0(K_n)] + \E[\nu_1(n-K_n)]+n.$$
Hence, with the notation (\ref{def_fi}) and $h(x)=x\log x$, $x\in[0,1]$, we have
\begin{align*}
\lefteqn{\frac 1 {\sqrt{n\log n}} \| \nu_0(K_n)+\nu_1(n-K_n)+n- \E[B_n^\mu] \|_3}  \\
&=\frac 1 {\sqrt{n\log n}} \| \nu_0(K_n)-\E[\nu_0(K_n)]+\nu_1(n-K_n)-\E[\nu_1(n-K_n)] \|_3 \\
&\leq \frac 1 {H \sqrt{n\log n}} \| h(K_n)-\E[h(K_n)]+h(n-K_n)-\E[h(n-K_n)] \|_3 \\
&\;\;\;\;~+ \frac 1 {\sqrt{n\log n}} \| f_0(K_n)-\E[f_0(K_n)] \|_3+\frac 1 {\sqrt{n\log n}} \| f_1(n-K_n)-\E[f_1(n-K_n)] \|_3
\end{align*}
An easy calculation reveals (details are given in the appendix, equation (\ref{ccc_app}))
\begin{align*}
\lefteqn{\| h(K_n)-\E[h(K_n)]+h(n-K_n)-\E[h(n-K_n)] \|_3 }\\
&=n\left\|  h\left( \frac{K_n}{n} \right)-\E \left[h\left(\frac{K_n}{n}\right)
\right]+h\left(\frac{n-K_n}{n}\right)-\E\left[ h\left(\frac{n-K_n}{n} \right)
\right] \right\|_3\\
&= \bo\left(n^{\nicefrac{1}{2}}\right).
 \end{align*}
The terms $\| f_0(K_n)-\E[f_0(K_n)] \|_3$ and $\| f_1(n-K_n)-\E[f_1(n-K_n)] \|_3$ are also of the order $\bo(n^{\nicefrac{1}{2}})$ by the argument used in
 (\ref{est_cond_app}).
 Altogether we have
$$\frac 1 {\sqrt{n\log n}} \| \nu_0(K_n)+\nu_1(n-K_n)+n- \E[B_n^\mu] \|_3= \bo\left( \frac{1}{\sqrt{\log n}}\right), $$
which implies (\ref{conv_2_app}).
\end{proof}

\pagebreak

\section*{Appendix}
\subsection*{Asymptotics of the Binomial and Poisson distribution}
The appendix is meant to cover some elementary asymptotic moment calculations of the binomial and Poisson distribution. These calculations were made for the sake of completeness
and may be removed in the published version of this paper.

The following approximations are immediate consequences of the concentration of the binomial distribution. Recall $x\log x=0$ for $x=0$.
\begin{lem}
\label{lemBinAs_app}
Let $p\in(0,1)$, $h(x):=x\log x$ for $x\in[0,1]$ and  $X_{n,p}$ be binomial $B(n,p)$ distributed for  $n\in\N$. Then we have as $n\rightarrow \infty$
\begin{align}
\E\left[ \log \left(\frac{X_{n,p}+1}{n+1} \right) - \log p \right] &= \bo\!\left(n^{-\nicefrac{1}{2}}\right), \label{aaa_app} \\
\left\| h(X_{n,p}/n) - \E\left[h(X_{n,p}/n)\right] \right\|_3& = \bo\left( n^{\nicefrac{-1}{2}}\right),
\label{ccc_app}\\
\E[h(X_{n,p}/n)-h(p)]&=\bo(n^{-2/3}). \label{eee_app}
\end{align}
\end{lem}
\begin{proof}
Proof of (\ref{aaa_app}): Note that we have for all $\varepsilon\in(0,1)$ by the mean value theorem
$$|\log (x)-\log(y)| \leq \varepsilon^{-1} |x-y|,\quad x,y\in[\varepsilon,1].$$
This yields
\begin{align*}
 &\left|\E\left[ \log \left(\frac{X_{n,p}+1}{n+1} \right) - \log p \right]\right|\\
&\leq\E\left[ \left|\log \left(\frac{X_{n,p}+1}{n+1} \right) - \log p\right|\Ind_{\{X_{n,p}\geq np/2\}} \right]
+\bo \left( \log n \Prob(X_{n,p}< np/2)\right)\\
&\leq \frac 2 p \E\left[\left| \frac{X_{n,p}+1-np-p}{n+1}\right|\right]+\bo \left( \log n \Prob(X_{n,p}< np/2)\right).
\end{align*}
The assertion follows since $\E[|(X_{n,p}-np) /\sqrt {n p (1-p)}|]$ converges to the first absolut moment of the standard
normal distribution and $\log n \Prob(X_{n,p}< np/2)=o(n^{-1/2})$ by Chernoff's bound.\\

\noindent
Proof of (\ref{ccc_app}): First note that $h$ is bounded on $[0,1]$ and that
we have for all $\varepsilon\in(0,1)$
$$|h^\prime (x) | \leq \log (1/\varepsilon)+1,\quad x\in[\varepsilon,1].$$
In particular, we obtain by the mean value theorem that
\begin{align}\label{lip_proof_ccc_app}
|h(x)-h(y)|\leq (\log (1/\varepsilon)+1) |x-y|, \quad x,y\in[\varepsilon,1].
\end{align}
With an independent copy $\widetilde{X}_{n,p}$ of $X_{n,p}$ we obtain by Jensen's inequality and (\ref{lip_proof_ccc_app})
\begin{align*}
&\left\| h(X_{n,p}/n) - \E\left[h(X_{n,p}/n)\right] \right\|_3^3\\
&= \E[(\E[h(X_{n,p}/n)-h(\widetilde{X}_{n,p}/n)|X_{n,p}])^3]\\
&\leq \E[(h(X_{n,p}/n)-h(\widetilde{X}_{n,p}/n))^3]\\
&= \E[(h(X_{n,p}/n)-h(\widetilde{X}_{n,p}/n))^3 \Ind_{\{X_{n,p},\widetilde{X}_{n,p}\in[n p/2, n]\}}] +\bo(\Prob(X_{n,p}\leq np/2))\\
&\leq (\log (2/p)+1)^3\E[(X_{n,p}/n-\widetilde{X}_{n,p}/n)^3]+\bo(\Prob(X_{n,p}\leq np/2))\\
&\leq \left(\frac{ \log (2/p)+1}{\sqrt{n}}\right)^3 (2\|X_{n,p}/\sqrt{n}\|_3)^3+\bo(\Prob(X_{n,p}\leq np/2)).
\end{align*}
The assertion follows by Chernoff's bound on $\Prob(X_{n,p}\leq np/2)$ and $\|X_{n,p}/\sqrt{n}\|_3\rightarrow \|N\|_3$ where
$N$ is $\mathcal{N}(0,p(1-p))$ distributed.\\


\noindent
Proof of (\ref{eee_app}): It is sufficient to show that
\begin{enumerate}
 \item{$h(p)-p\E[\log (X_{n,p}/n) \Ind_{\{X_{n,p}\geq 1\}}]=\bo(n^{-2/3})$,}
\item{$\E[h(X_{n,p}/n)-p\log(X_{n,p}/n)\Ind_{\{X_{n,p}\geq 1\}}]=\bo(n^{-2/3})$.}
\end{enumerate}
For the first part note that we have
\begin{align*}
 &\left|h(p)-p\E[\log (X_{n,p}/n) \Ind_{\{X_{n,p}\geq 1\}}]\right| \\
&= p\left| \E\left[\log\left( \frac {X_{n,p}} {np}\right)\Ind_{\{X_{n,p}\geq 1\}}\right] \right|+\bo\left((1-p)^n\right)\\
&=p\left|\E\left[\left(\log\left( 1+\frac {X_{n,p}-np} {np}\right)-\frac{X_{n,p}-np} {np}\right)\Ind_{\{X_{n,p}\geq 1\}}\right]\right|
+\bo\left((1-p)^n\right)\\
&\leq p\left|\E\left[\left(\log\left( 1+\frac {X_{n,p}-np} {np}\right)-\frac{X_{n,p}-np} {np}\right)\Ind_{\{|X_{n,p}-np|\leq n^{2/3}\}}\right]\right| \\
&\quad+(\log(np)+1/p)\Prob(|X_{n,p}-np|> n^{2/3})+\bo\left((1-p)^n\right).
\end{align*}
Since we have $\log(1+x)-x = \bo(x^2)$ for $x\rightarrow 0$ and $\Prob(|X_{n,p}-np|> n^{2/3})=o(n^{-1})$ by Chernoff's bound,
we may conclude that
$$h(p)-p\E[\log (X_{n,p}/n) \Ind_{\{X_{n,p}\geq 1\}}]=\bo(n^{-2/3}).$$
In order to obtain the second bound, note that
\begin{align*}
 &\E[h(X_{n,p}/n)-p\log(X_{n,p}/n)\Ind_{\{X_{n,p}\geq 1\}}]\\
&=\E\left[\left(h(X_{n,p}/n)-p\log(X_{n,p}/n)\right)\Ind_{\{X_{n,p}\geq 1\}}\right]+\bo\left((1-p)^n\right)\\
&=\frac 1 {\sqrt n} \E\left[\frac {X_{n,p}-np} {\sqrt n} \log\left(\frac {X_{n,p}} n \right) \Ind_{\{X_{n,p}\geq 1\}}\right]
+\bo\left((1-p)^n\right)\\
&=\frac 1 {\sqrt n} \E\left[\frac {X_{n,p}-np} {\sqrt n} \log\left(\frac {X_{n,p}} n \right) \Ind_{\{|X_{n,p}-np|\leq n^{2/3}\}}\right]
+o\left(n^{-2/3}\right)\\
&=\frac 1 {\sqrt n} \E\left[\frac {X_{n,p}-np} {\sqrt n} \log(p) \Ind_{\{|X_{n,p}-np|\leq n^{2/3}\}}\right]\\
&\quad +\frac 1 {\sqrt n} \E\left[\frac {X_{n,p}-np} {\sqrt n} \log\left(1+\frac{X_{n,p}-np} {np}\right) \Ind_{\{|X_{n,p}-np|\leq n^{2/3}\}}\right]
+o\left(n^{-2/3}\right).
\end{align*}
Since $\log(1+x)=\bo(x)$ as $x\rightarrow 0$ and $\E[|(X_{n,p}-np)/\sqrt n|]$ converges to the first absolute moment of the
$\mathcal{N}(0,p(1-p))$ distribution, we obtain for the second summand
$$\frac 1 {\sqrt n} \E\left[\frac {X_{n,p}-np} {\sqrt n} \log\left(1+\frac{X_{n,p}-np} {np}\right) \Ind_{\{|X_{n,p}-np|\leq n^{2/3}\}}\right]
=\bo(n^{-5/6}).$$
For the first summand note that $\E[( X_{n,p}-np)/ \sqrt n]=0$ which implies
\begin{align*}
&\frac 1 {\sqrt n} \E\left[\frac {X_{n,p}-np} {\sqrt n} \log(p) \Ind_{\{|X_{n,p}-np|\leq n^{2/3}\}}\right]\\
&=-\frac 1 {\sqrt n} \E\left[\frac {X_{n,p}-np} {\sqrt n} \log(p) \Ind_{\{|X_{n,p}-np|> n^{2/3}\}}\right]\\
&=\bo(\Prob(|X_{n,p}-np|> n^{2/3}))\\
&=o(n^{-2/3}).
\end{align*}
Hence, we obtain $\E[h(X_{n,p}/n)-p\log(X_{n,p}/n)\Ind_{\{X_{n,p}\geq 1\}}]=\bo(n^{-2/3})$ which combined with the
first result yields the assertion.
\end{proof}

The next Lemma provides asymptotic results for the poisson distribution that are needed for the analysis of the variance:
\begin{lem}\label{lemPoisAs}
 For $\lambda>0$ let $N_\lambda$ be Poisson($\lambda$) distributed.
Then we have for all $\alpha,\beta>0$ as $\lambda\rightarrow\infty$
\begin{align*}
\E[N_\lambda^\alpha]&=\bo(\lambda^{\alpha}),\\
 \E\left[N_\lambda^\alpha (\log N_\lambda)^\beta\right]&=\bo\left(\lambda^\alpha(\log \lambda)^\beta\right).
\end{align*}
\end{lem}
\begin{proof}
 We start with the analysis of $\E[N_\lambda^\alpha]$: For $\alpha\in\N$ the assertion follows by induction and the fact that
for every $n\in\N_0$ we have
$$\E\left[\prod_{i=0}^n (N_\lambda -i)\right]=\lambda^{n+1}.$$
For $\alpha\in(0,1)$ note that $x\mapsto x^{\alpha}$ is concave on $[0,\infty)$ and therefore,
by Jensen's inequality
\begin{align*}
 \E[N_\lambda^\alpha]\leq \left(\E[N_\lambda]\right)^\alpha =\lambda^\alpha.
\end{align*}
Finally, for $\alpha\in(1,\infty) \cap \N^c$ we have that $x\mapsto x^{\alpha / \lceil\alpha \rceil}$ is concave on $[0,\infty)$
which yields
\begin{align*}
 \E[N_\lambda^\alpha] \leq (\E[N_\lambda^{\lceil \alpha \rceil}])^{\alpha/ \lceil \alpha \rceil}
\end{align*}
and the assertion follows by the results for $\alpha\in\N$.\\

\noindent
For the second part of the proof we use the following decomposition
\begin{align*}
  \E\left[N_\lambda^\alpha (\log N_\lambda)^\beta\right]&= \E\left[N_\lambda^\alpha (\log N_\lambda)^\beta\Ind_{\{N_\lambda\leq \lambda^{\alpha+1}\}}\right]
+ \E\left[N_\lambda^\alpha (\log N_\lambda)^\beta\Ind_{\{N_\lambda>\lambda^{\alpha+1}\}}\right]\\
&\leq (\alpha+1)^\beta (\log \lambda)^\beta\E[N_\lambda^\alpha]+\E\left[N_\lambda^\alpha (\log N_\lambda)^\beta\Ind_{\{N_\lambda>\lambda^{\alpha+1}\}}\right]\\
&=\bo(\lambda^\alpha(\log \lambda)^\beta)+\E\left[N_\lambda^\alpha (\log N_\lambda)^\beta\Ind_{\{N_\lambda>\lambda^{\alpha+1}\}}\right],
\end{align*}
where the last step holds since $\E[N_\lambda^\alpha]=\bo(\lambda^\alpha)$.
Hence, it is sufficient to show that
$$\E\left[N_\lambda^\alpha (\log N_\lambda)^\beta\Ind_{\{N_\lambda>\lambda^{\alpha+1}\}}\right]=\bo(\lambda^\alpha).$$
Since we have $n^\alpha (\log n )^\beta \leq C_{\alpha\beta} n^{3\alpha/2}$ for a sufficiently large constant $C_{\alpha\beta}$
and all $n\in\N_0$, we obtain
\begin{align*}
 \E\left[N_\lambda^\alpha (\log N_\lambda)^\beta\Ind_{\{N_\lambda>\lambda^{\alpha+1}\}}\right]&\leq
C_{\alpha\beta}\E\left[N_\lambda^{3\alpha/2}\Ind_{\{N_\lambda>\lambda^{\alpha+1}\}}\right]\\
&\leq C_{\alpha\beta}\sqrt{\E[N_\lambda^{3\alpha}] \Prob(N_\lambda>\lambda^{\alpha+1})}
\end{align*}
where the last inequality holds by the Cauchy-Schwarz inequality. Together with the previous result $\E[N_\lambda^{3\alpha}]=\bo(\lambda^{3\alpha})$
and Markov's inequality this yields
$$\E\left[N_\lambda^\alpha (\log N_\lambda)^\beta\Ind_{\{N_\lambda>\lambda^{\alpha+1}\}}\right]=\bo(\lambda^{\alpha})$$
and the assertion follows.
\end{proof}

\end{document}